\documentclass{siamart171218}
\usepackage{float}
\usepackage{diagbox}
\newtheorem{cor}{Corollary}

\newcommand{\be}{\begin{equation}}
\newcommand{\ee}{\end{equation}}

\newcommand{\bx}{\mathbf{x}}
\newcommand{\by}{\mathbf{y}}

\newcommand{\Dx}{\Delta x}

\newcommand{\bu}{\mathbf{u}}

\newcommand{\bw}{\mathbf{w}}

\newcommand{\Dt}{\Delta t}

\newcommand{\dx}{\Delta x}
\newcommand{\dt}{\Delta t}
\newcommand{\aij}{\alpha_{i,j}}
\newcommand{\bij}{\beta_{i,j}}

\newcommand{\m}[1]{\mathbf{#1}}
\newcommand{\mA}{\m{A}}

\newcommand{\mAh}{\hat{\m{A}}}
\newcommand{\mD}{\m{D}}
\newcommand{\mS}{\m{S}}
\newcommand{\mT}{\m{T}}
\newcommand{\mR}{\m{R}}
\newcommand{\mP}{\m{P}}

\newcommand{\mI}{\m{I}}

\newcommand{\mzero}{\m{0}}
\renewcommand{\v}[1]{\boldsymbol{#1}}
\newcommand{\transpose}{^\mathrm{T}}

\newcommand{\vb}{\v{b}}
\newcommand{\vc}{\v{c}}
\newcommand{\ve}{\v{e}}
\newcommand{\vu}{\v{u}}
\newcommand{\vl}{\boldsymbol{l}}

\newcommand{\vf}{\v{f}}

\newcommand{\sspcoef}{\mathcal{C}}

\newcommand{\ceff}{\sspcoef_{\textup{eff}}}
\newcommand{\DtFE}{\Dt_{\textup{FE}}}
\newcommand{\tDtFE}{\tilde{\Dt}_{\textup{FE}}}

\newcommand{\ste}{\boldsymbol{\tau}}

\newcommand{\btheta}{\boldsymbol{\theta}}

\newcommand{\bff}{\mathbf{f}}

\newcommand{\bb}{\mathbf{b}}
\newcommand{\bbh}{\hat{\mathbf{b}}}

\newcommand{\mC}{\m{C}}

\renewcommand{\v}[1]{\mathbf{#1}}

\title{Strong Stability Preserving Integrating Factor Two-step Runge--Kutta Methods}

\author{%
Leah Isherwood\thanks{Mathematics Department, University of Massachusetts Dartmouth, 285 Old Westport Road,
North Dartmouth MA 02747.} \and
Zachary J. Grant\thanks{Department of Computational and Applied Mathematics, Oak Ridge National Laboratory, Oak Ridge TN 37830.}
\and 
Sigal Gottlieb\footnotemark[1]
}

\begin{document}
\maketitle

%\begin{keywords}
%\end{keywords}

\bibliographystyle{siam}

\begin{abstract} 
Problems  with components that feature significantly different time scales, where
the stiff time-step restriction comes from a linear component, implicit-explicit (IMEX) methods alleviate 
this restriction if the concern is linear stability. However, when nonlinear non-inner-product stability
properties are of interest, such as in the evolution of hyperbolic
partial differential equations with shocks or sharp gradients,  
linear inner-product stability is no longer sufficient for convergence,
and so strong stability preserving (SSP)   methods are often needed. 
Where the SSP property is needed,  IMEX SSP Runge--Kutta (SSP-IMEX) methods have very restrictive time-steps. 
An alternative to  SSP-IMEX schemes is to adopt an integrating factor approach to handle the linear component exactly 
and step the transformed problem forward using some time-evolution method.
The strong stability properties of integrating factor Runge--Kutta methods were established in 
\cite{SSPIFRK-SINUM},
where it was shown that it is possible to define explicit integrating factor Runge--Kutta methods that 
preserve strong stability properties satisfied by each of the two components when coupled with 
forward Euler time-stepping.  It was proved that the solution will be SSP  if the transformed problem is stepped 
forward with an explicit SSP Runge--Kutta method that has non-decreasing abscissas. However,
explicit SSP Runge--Kutta methods have an order barrier of $p=4$, and sometimes higher order is desired.
In this work we consider explicit SSP two-step Runge--Kutta integrating factor methods to raise the order. 
We show that strong stability is ensured if the two-step Runge--Kutta 
method used to evolve the transformed problem is SSP and has non-decreasing abscissas.
We find such methods up to eighth order  and present  their SSP coefficients. Adding a step allows us to break 
the fourth order barrier on explicit SSP  Runge--Kutta methods;
furthermore, our explicit SSP two-step Runge--Kutta methods with non-decreasing abscissas typically 
have larger SSP coefficients than the corresponding
one-step methods.
A selection of our methods are tested for convergence and demonstrate the design order. 
We also show, for selected methods, that the SSP time-step predicted by the theory is a lower bound of the 
allowable time-step for linear and nonlinear problems that satisfy the total variation diminishing (TVD)  condition. 
We compare some of the non-decreasing abscissa SSP two-step Runge--Kutta methods to previously found 
methods that do not satisfy this criterion on linear and nonlinear TVD  test cases to show that this  non-decreasing 
abscissa  condition is indeed necessary in practice as well as theory. 
We also compare these results to our SSP integrating factor Runge--Kutta methods designed in \cite{SSPIFRK-SINUM}.
 
\end{abstract}

{\em This paper is dedicated to the memory of Saul Abarbanel. His wisdom, humor, and kindness were a gift to all who knew him.}

\section{Introduction\label{sec:intro}}

The behavior of the numerical solution  of a hyperbolic partial differential  equation (PDE) of the form
\begin{eqnarray}
\label{PDE}
U_t +f(U)_x = 0,
\end{eqnarray}
depends on properties of the spatial  discretization and of the  time discretization.  
When the solution is smooth, stability is guaranteed 
by analyzing the $L_2$ stability properties of the discretization applied to the linear problem.
However, when dealing with a non-smooth solution, the numerical solution may contain non-physical oscillations that prevent the approximation from converging uniformly and thus
 $L_2$ linear stability  is not sufficient to ensure convergence \cite{LeVequeBook}.  
To ensure that the numerical solution does not form stability-destroying oscillations, we require that
the numerical method satisfies nonlinear non-inner-product stability properties such as 
stability in the maximum norm or in the total variation (TV) semi-norm.

For nonlinear hyperbolic problems with discontinuous solutions we must, 
therefore, analyze the nonlinear non-inner-product stability properties
of a highly nonlinear complex spatial discretization combined with a high order time discretization. 
Instead of this difficult  task a method-of-lines formulation is generally followed: we develop a  spatial discretization that satisfies nonlinear non-inner-product stability properties when coupled with the forward Euler time stepping method. Next, we use a high order strong stability preserving (SSP) time discretization \cite{shu1988b, shu1988,ruuth2001,SpiteriRuuth2002,hundsdorfer2003,ketcheson2009, ketcheson2008, tsrk, SSPbook2011, msrk} which preserves the properties of the spatial discretization coupled with forward Euler.
Explicit strong stability preserving (SSP) Runge--Kutta methods were developed in \cite{shu1988b,shu1988}
to preserve the properties of total variation diminishing (TVD) spatial discretizations for 
hyperbolic conservation laws \eqref{PDE} with discontinuous solutions. They have since 
been developed extensively and have been widely used to preserve different numerical stability properties
needed in a variety of  application areas. 
Furthermore, many classes of SSP time-stepping methods have been studied, including 
SSP explicit and implicit linear multi-step methods \cite{SSPbook2011}, 
Runge--Kutta methods \cite{shu1988b,shu1988,ketcheson2008,ketcheson2009}, 
multi-stage multi-step methods \cite{Sandu,tsrk,msrk}, 
and multi-stage multi-derivative methods \cite{Grant1,Grant2}.

\subsection{SSP methods\label{SSPintro}}
The key to developing SSP time-stepping methods is ensuring that the methods can be re-written as convex combinations of forward Euler steps. To illustrate this concept, we begin with the PDE \eqref{PDE} above, and 
use a   spatial discretizations of $f(U)_x$ that ensures that when we evolve the resulting semi-discretized system of ordinary differential equations (ODEs)
\begin{eqnarray}
\label{ode}
u_t = F(u),
\end{eqnarray}
%(where $u$ is a vector of approximations to $U$,  $u_j \approx U(x_j) $) 
using the forward Euler  method, the  strong stability property in the convex functional $\| \cdot \|$ is satisfied
\begin{eqnarray} \label{FEstrongstability}
\|u^{n+1} \| = \| u^n + \dt F(u^{n}) \| \leq \| u^n \| ,
\end{eqnarray}
under a step size restriction
\begin{eqnarray} \label{FEcond}
0 \leq \dt \leq \DtFE.
\end{eqnarray}
% \cite{shu1988, shu1988b}. 
%These works studied total variation diminishing (TVD) spatial discretizations
%that can handle discontinuities. 
%
%developed for the time evolution of
Next, we use a higher order time integrator that can be written as a convex combination of forward Euler steps,
so that we ensure that any convex functional strong stability property that is satisfied 
by the forward Euler method will still be satisfied by the higher order time discretization, perhaps under a different time-step.
For example, an  $s$-stage explicit Runge--Kutta method (denoted eSSPRK$(s,p)$ where $p$ is the order)
can be written in Shu-Osher form \cite{SSPbook2011}:
\begin{eqnarray}
\label{rkSO}
u^{(0)} & =  & u^n, \nonumber \\
u^{(i)} & = & \sum_{j=0}^{i-1} \left( \aij u^{(j)} +
\dt \bij F(u^{(j)}) \right), \; \; \; \; i=1, . . ., s\\% \quad \aik \geq 0,  \qquad i=1 ,..., m \\
 u^{n+1} & = & u^{(s)} . \nonumber
\end{eqnarray}
(Where $\sum_{j=0}^{i-1} \aij =1$ is required for consistency).
If  $\aij $ and $\bij$ are non-negative, and any  $\aij$ is zero only if its corresponding $\bij$ is zero,
then we can rearrange each stage  into a convex combination of forward Euler steps:
\[ \| u^{(i)}\|  =  
\left\| \sum_{j=0}^{i-1} \left( \aij u^{(j)} + \dt \bij F(u^{(j)}) \right) \right\|
=  \left\|  \sum_{j=0}^{i-1} \aij  \left(  u^{(j)} + \dt \frac{\bij}{\aij} F(u^{(j}) \right) \right\|    .\]  
Clearly, then,  \eqref{FEstrongstability}  tells us that 
 \[ \| u^{(i)}\|   \leq   \sum_{j=0}^{i-1} \aij  \, \left\| u^{(j)} + \dt \frac{\bij}{\aij} F(u^{(j}) \right\|  
 \leq  \|u^n\| , \]
while \eqref{FEcond} imposes the time-step restriction
\begin{eqnarray}
\dt \leq \min_{i,j} \frac{\aij}{\bij} \DtFE.
\end{eqnarray}
(We use the convention that  the ratio in considered infinite in the case  where a $\bij$ is equal to zero). 
This convex combination ensures that the internal stages also satisfy the strong stability property
under the same time-step restriction; 
this can be  important  when pressure, density, or water height, are simulated, because in these cases
a negative value at an intermediate or final stage may prevent  the simulation from continuing \cite{HesthavenCLbook}. 
Clearly, the convex combination condition is a sufficient condition for 
strong stability preservation for explicit Runge--Kutta methods;
in \cite{SSPbook2011, kraaijevanger1991,spijker2007} it was shown to be a  necessary condition as well.

%The last inequality above  follows from  the strong stability conditions \eqref{FEstrongstability} and  \eqref{FEcond} 
%\[ \left\| u^{(j)} + \dt F(u^{(j})  \right\|   \leq \left\| u^{(j)}   \right\|   \; \; \; \forall \dt \leq \DtFE \]  and 
%the consistency condition $\sum_{j=0}^{i-1} \aij =1$. 
This shows that when a higher order time discretization method is written as a convex combination
of forward Euler steps,  it will  {\em preserve}  the forward Euler condition \eqref{FEstrongstability}, 
under a modified time-step restriction $\Dt \le \sspcoef \DtFE$. 
When $ \sspcoef  >0$, the method is called {\em strong stability preserving} (SSP) with {\em SSP coefficient}
 $\sspcoef$ \cite{shu1988b}. While in the original papers  \cite{shu1988b, shu1988}, 
the term $\| \cdot \|$ was the total variation semi-norm, the strong stability preservation property holds
for any convex functional  $\| \cdot \|$ as long as the spatial discretization satisfies \eqref{FEstrongstability} in that convex functional for some $\DtFE$ as in \eqref{FEcond}.

Not all methods can be decomposed into convex combinations of forward Euler steps with
$\sspcoef >0$.  Explicit SSP Runge--Kutta methods  cannot exist for order $p>4$    \cite{kraaijevanger1991,ruuth2001}. 
Furthermore, the SSP coefficient $\sspcoef$ is restricted too, as
all explicit $s$-stage Runge--Kutta methods have a SSP bound $\sspcoef \leq s$ \cite{SSPbook2011}. 
Unlike for smooth problems, where implicit methods  (or implicit treatment of stiff terms)
can eliminate the time-step restriction needed for linear or nonlinear inner-product norm stability,
a strong stability preserving general linear method (GLM) of order  greater than one has a finite SSP time-step \cite{spijker1983}.  
In fact, it has been shown  \cite{lenferink1991,ferracina2008,ketcheson2009,SSPIMEX} that
the SSP  time-step restrictions for implicit and implicit-explicit (IMEX) SSP Runge--Kutta methods have a restrictive  
observed  bound of $\sspcoef \leq 2s$ .
The limitation on the SSP coefficient of implicit and IMEX SSP Runge--Kutta methods led to the study of 
SSP integrating factor Runge--Kutta methods in \cite{SSPIFRK-SINUM}.

\subsection{Overview of current paper}

The explicit SSP integrating factor Runge--Kutta methods and the associated theory developed in our prior work \cite{SSPIFRK-SINUM}, reviewed in Section \ref{sec:Review-IFRK}, allow us to alleviate the time-step restriction while preserving the strong stability property.
However, because  explicit SSP Runge--Kutta methods have an order barrier of $p\leq 4$, we cannot hope to
get higher order  explicit SSP integrating factor Runge--Kutta  methods. We could try to work with 
an integrating factor multi-step approach, which is SSP as long as the multi-step method is SSP (as we 
show in Subsection \ref{sec:multistepif}) but these methods have small SSP coefficients and require many steps for high order.
Instead, we follow the process in \cite{tsrk, msrk} and examine the SSP properties of integrating factor methods
based on explicit SSP two-step Runge--Kutta (eSSP-TSRK) methods such as those in \cite{tsrk}.
When these are used as a basis for integrating factor TSRK methods, the result is SSP provided that the abscissas are non-decreasing. 
In Section \ref{sec:background} we discuss explicit SSP two-step Runge--Kutta methods, review the optimization problem, 
and review the resulting  eSSP-TSRK methods presented in \cite{tsrk}. In Section \ref{sec:SSPIF} 
we develop the SSP theory for explicit SSP integrating factor two-step Runge--Kutta (SSPIF-TSRK) methods, 
which is very similar to that of the explicit SSP integrating factor 
Runge--Kutta (SSPIFRK) methods in \cite{SSPIFRK-SINUM}, 
and show that as long as these methods have non-decreasing abscissas the result is SSP. 
We show that  to find appropriate explicit SSP two-step Runge--Kutta for pairing with integrating factor methods,  
the optimization problem needs to be augmented by a simple condition on the abscissas. 
In Section \ref{sec:optimal} we present the optimized methods found using this approach, 
and discuss their features. In Section \ref{sec:test} we test these methods for convergence 
and for their SSP performance on numerical test cases.

\subsection{Efficient computation of the matrix exponential}
 It is important to note that in the current work as well as our prior work 
 \cite{SSPIFRK-SINUM,SSPIFRK-downwind}, 
 the cost of  computation of the  matrix exponential will be a critical factor in the ability 
 to efficiently implement  the proposed methods.  If the computation of the matrix exponential  
 is only needed once per simulation  (i.e. if $L$ is a constant coefficient operator), 
 the cost may be reasonable, but in some cases the exponential must be computed at every step,
and low-storage, matrix-free approaches are needed. New approaches to efficiently compute the
matrix exponential have been recently considered in 
\cite{AlMohyHigham,Sidje, NiesenWright, GaudreaultRainwaterTokman}, 
and such approaches, as well as others, will be critical for bringing the integrating factor 
methods proposed in \cite{SSPIFRK-SINUM,SSPIFRK-downwind} and the current work into practical use.

 \section{Review of SSP integrating factor Runge--Kutta methods}\label{sec:Review-IFRK}

 %\section{Review of SSP integrating factor Runge--Kutta methods}
 
 In \cite{SSPIFRK-SINUM,SSPIFRK-downwind} we considered  problems of the form
 \begin{eqnarray}\label{IFproblemform}
u_t = Lu + N(u)  
\end{eqnarray}
that result from a semi-discretization of a hyperbolic PDE.
We focus on the case where the problem has a linear constant coefficient component
 $L u $ and a nonlinear component $N(u)$, such  that
%The case we are interested in is when some strong stability condition is known for the forward Euler step of the 
%nonlinear component $N(u)$
\begin{eqnarray}  \label{nonlinearFEcond}
\| u^n + \dt N(u^n) \| \leq \|u^n\| \qquad \mbox{for} \qquad \dt \leq \DtFE
\end{eqnarray}
%while taking a forward Euler step using the linear component $Lu$ results in the  strong stability condition 
and
\begin{eqnarray} \label{linearFEcond}
\| u^n + \dt L u^n \| \leq \|u^n\| \qquad \mbox{for} \qquad \dt \leq \tDtFE.
\end{eqnarray}
The notation  $\| \cdot \|$ here refers to some convex functional (not usually a norm)
needed for nonlinear non-inner-product stability.
Of particular interest is the case where $L$ is a linear operator that significantly restricts the 
allowable time-step due to strong stability concerns, i.e. where $\tilde{\DtFE} <<  {\DtFE} $. 
For such cases, where nonlinear non-inner-product stability properties are of concern, 
an implicit or implicit-explicit (IMEX) SSP scheme doesn't significantly alleviate the allowable time-step \cite{SSPIMEX}
and such methods will result in severe constraints 
 on the allowable time-step \cite{lenferink1991,ketcheson2009,ferracina2008,SSPbook2011,SSPIMEX}.

This motivated our investigation of integrating factor methods, where the 
 linear component $L u$ is handled exactly and then 
 the time-step restriction is dependent only upon the step size restriction $ \dt \leq {\DtFE}$ 
 coming from the  the nonlinear component $N(u)$. 
The methodology we considered in \cite{SSPIFRK-SINUM}  in order to  alleviate the restriction on the allowable time-step begins with an integrating factor approach:
\begin{eqnarray*}
e^{-L t } u_t - e^{-L t }  Lu = e^{-L t } N(u)  \\
\left( e^{-L t } u \right)_t = e^{-L t } N(u).
\end{eqnarray*}
Defining  $w = e^{-L t } u $ gives us  the modified ODE system
\begin{eqnarray} \label{IF_ODE}
w_t  = e^{-L t } N(e^{L t } w)   = G(w),
\end{eqnarray}
which can then be evolved forward in time. In  \cite{SSPIFRK-SINUM}  we used an
explicit SSP Runge--Kutta method of the form \eqref{rkSO}. We observed that this approach is not sufficient to
ensure strong stability of the system. However,  as we showed, if the transformed problem 
is stepped forward using a SSP Runge--Kutta method that is carefully chosen to have
non-decreasing abscissas the resulting method \textit{will} preserve the desired  strong stability property.

\subsection{Motivating example} \label{motivating}
Before we repeat the formal results from \cite{SSPIFRK-SINUM}, which explain how non-decreasing abscissas in the Runge--Kutta method are needed to preserve strong stability when using an integrating factor approach, we present the following example, presented in
\cite{SSPIFRK-SINUM}, which illustrates this fact:

Consider the partial differential equation
 \begin{align} \label{BurgersAdvection}
u_t + a u_x +  \left( \frac{1}{2} u^2 \right)_x & = 0 \hspace{.75in}
    u(0,x)  =
\begin{cases}
1, & \text{if } 0 \leq x \leq 1/2 \\
0, & \text{if } x>1/2 
\end{cases}
\end{align}
where $x \in [0,1]$ and the boundary conditions are periodic. In this example we consider $a=10$.
A first-order upwind difference with $400$ points in space is used to semi-discretize the linear term $L u \approx -10 u_x$.  
The nonlinear terms $N(u) \approx - \left( \frac{1}{2} u^2 \right)_x $ are approximated using a  fifth order WENO 
finite difference \cite{WENO}.

The transformed problem \eqref{IF_ODE} is stepped forward in time by the explicit SSP Runge--Kutta method with non-decreasing abscissas,
denoted eSSPRK$^+(3,3)$  in  \cite{SSPIFRK-SINUM}, which has  $\sspcoef = \frac{3}{4}$. The resulting SSP integrating factor
Runge--Kutta method is:
\begin{eqnarray}     \label{SSPIFRK33}
     u^{(1)} &= & \frac{1}{2}e^{\frac{2}{3} \dt L} u^n +  \frac{1}{2} e^{\frac{2}{3} \dt L}  \left(u^{n}+\frac{4}{3}\dt N(u^n)\right) \nonumber \\
     u^{(2)} &= & \frac{2}{3}  e^{\frac{2}{3} \dt L} u^n + \frac{1}{3} \left(u^{(1)} + \frac{4}{3} \dt N(u^{(1)})\right) \nonumber \\
     u^{n+1} & = & \frac{59}{128} e^{ \dt L} u^n + \frac{15}{128} e^{ \dt L}  \left(u^n +  \frac{4}{3} \dt N(u^n)\right)  \\
     && + \frac{27}{64} e^{\frac{1}{3} \dt L} \left(u^{(2)} 			+ \frac{4}{3} \dt N(u^{(2)})\right). \nonumber
\end{eqnarray}
For comparison, the transformed problem \eqref{IF_ODE} is also stepped forward in time by the  
explicit eSSPRK(3,3) Runge--Kutta method  in \cite{shu1988b} 
which has  $\sspcoef=1$ (this method is widely known as the Shu-Osher method). 
The resulting (non-SSP) integrating factor Runge--Kutta method becomes:
\begin{eqnarray} \label{SOIF}
     u^{(1)} &= & e^{L\dt}u^n + e^{L\dt} \dt N(u^n) \nonumber \\
     u^{(2)} &= & \frac{3}{4} e^{\frac{1}{2}L\dt} u^n + \frac{1}{4} e^{-\frac{1}{2}L\dt} \left(u^{(1)} + \dt N(u^{(1)})\right)  \nonumber \\
     u^{n+1} & = & \frac{1}{3}e^{L\dt}  u^n +\frac{2}{3} e^{\frac{1}{2}L\dt}  \left(u^{(2)} +  \dt N(u^{(2)})\right). 
\end{eqnarray}
Exponentials with negative exponents appear in this formulation  due to the fact that the SSP Shu-Osher method 
it is based on has decreasing abscissas.
Using different values of $\dt$ we  evolved the solution 25 time steps using 
the integrating factor Runge--Kutta methods \eqref{SSPIFRK33} and \eqref{SOIF},
and calculated the maximal rise in total variation over each stage for  all  time steps. 

\begin{figure}[h]
\centering
 \label{fig:motivating} \includegraphics[scale=.375]{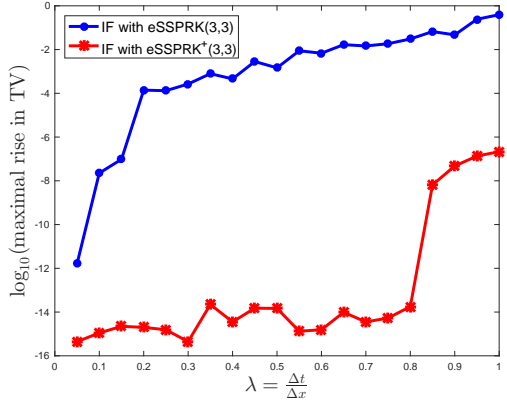} 
    \caption{ Motivating example: TV behavior of the evolution over 25 time steps using the 
     integrating factor methods \eqref{SOIF} (blue dots) and \eqref{SSPIFRK33} (red stars).
%On the x-axis is the value of $\lambda = \frac{\dt}{\dx}$, on the y-axis is $log_{10}$ of the maximal rise in TV.
}
\end{figure}

In Figure \ref{fig:motivating} we show the $\log_{10}$ of the maximal rise in total variation vs. the value of $\lambda = \frac{\dt}{\dx}$ of the evolution 
using \eqref{SOIF} (in blue) and  using \eqref{SSPIFRK33} (in red). The results from 
 \eqref{SSPIFRK33} maintain a small maximal rise in total variation up to $\lambda \approx 0.8$,  while 
in the results from \eqref{SOIF} we observe that the  maximal rise in total variation  is large even for very 
small values of $\lambda$. This behavior is due to the fact that basing an integrating factor  Runge--Kutta 
method on an explicit SSP Runge--Kutta method  is not enough to ensure the preservation of a strong stability 
property -- the method must also satisfy the non-decreasing abscissa  condition  to ensure that the strong 
stability property is preserved. 

\subsection{SSP analysis of integrating factor Runge--Kutta method}
To understand the results in \cite{SSPIFRK-SINUM}  recall that each stage of the explicit Runge--Kutta method  \eqref{rkSO}  
applied to the transformed problem \eqref{IF_ODE} becomes
\begin{eqnarray*}
u^{(i)}   &=&    \sum_{j=0}^{i-1} \left( \aij e^{L (t_i-t_j) }  u^{(j)} + \dt \bij e^{L ( t_i-t_j)}  N(u^{(j)}) \right) \\
 &=&  \sum_{j=0}^{i-1} \left( \aij e^{L (c_i-c_j) \dt }  u^{(j)} + \dt \bij e^{L ( c_i-c_j)\dt}  N(u^{(j)}) \right)\\
 &=&  \sum_{j=0}^{i-1} e^{L (c_i-c_j) \dt } \left( \aij u^{(j)} + \dt \bij N(u^{(j)}) \right)  . 
 \end{eqnarray*}
 In the following, we consider each piece of this formula. Each of the following results is taken directly from
 \cite{SSPIFRK-SINUM} and the proofs appear in that work. We first consider the behavior of the exponential operator $e^{\tau L}$. 
 The following theorem ensures that this term is strongly stable as long as $\tau \geq 0$:
\begin{theorem} \label{thm:exp_FEcond}
If  a  linear operator $L$  satisfies  \eqref{linearFEcond} for some value of $\tDtFE > 0 $,
then 
\begin{equation} \label{EXPcondition}
\| e^{\tau L} u^n \|  \leq \| u^n \|   \; \; \; \; \forall \; \tau \geq 0 . 
\end{equation}
\end{theorem}
Note that in the following results we no longer require the condition \eqref{linearFEcond} but rather the weaker condition \eqref{EXPcondition}.
Pairing this exponential with a forward Euler type step still preserves strong stability:
\begin{cor}
Given a  linear operator $L$ that satisfies  \eqref{EXPcondition}
and a (possibly nonlinear) operator $N(u)$ that satisfies \eqref{nonlinearFEcond}  
for some value of ${\Delta t}_{FE} > 0 $, we have 
\begin{equation}
\| e^{\tau L} ( u^n  + \dt N (u^n )) \|  \leq \| u^n \|   \; \; \; \; \forall \dt \leq \DtFE,  \; \; \; \mbox{provided that} \; \; \tau \geq 0. 
\end{equation}
\end{cor}
Finally, we put the pieces together:
\begin{theorem} \label{thm:SSPIF} 
Given a  linear operator $L$ that satisfies \eqref{EXPcondition}
and a (possibly nonlinear) operator $N(u)$ that satisfies \eqref{nonlinearFEcond}  
for some value of ${\Delta t}_{FE} > 0 $, and a Runge--Kutta integrating factor method of the form
\begin{eqnarray} \label{rkIFSO}
u^{(0)} & =  & u^n, \nonumber \\
u^{(i)} & = & \sum_{j=0}^{i-1} e^{L (c_i-c_j) \dt }  \left( \aij u^{(j)} + \dt \bij N(u^{(j)}) \right), \; \; \; \; i=1, . . ., s\\
 u^{n+1} & = & u^{(s)}  \nonumber
\end{eqnarray}
where $0=c_1 \leq c_2 \leq  . . . \leq c_s$, then $u^{n+1}$ obtained from \eqref{rkIFSO} satisfies
\begin{equation}
\|u^{n+1}\| \leq \|u^n\| \; \; \; \forall \dt \leq \sspcoef \DtFE.
\end{equation}
where 
\begin{equation*}
\sspcoef  =  \min_{i,j} \frac{\aij}{\bij}.
\end{equation*}
\end{theorem}
In \cite{SSPIFRK-SINUM} we used these results to motivate the search for explicit SSP Runge--Kutta methods with non-decreasing abscissas,
that would pair well with an integrating factor approach to produce SSPIFRK methods.
We found explicit SSP Runge--Kutta methods of up to $s=10$ stages and  order $p=4$
and demonstrated the performance of the corresponding explicit SSP integrating factor Runge--Kutta 
methods on widely used test cases.

%In Section \ref{sec:conclusions} we conclude that  the newly developed SSP theory for integrating factor 
%Runge--Kutta methods provides a provable bound on the allowable time-step   (which is often sharp in practice)
%for preservation of the nonlinear non-inner-product stability properties. 

In \cite{SSPIFRK-downwind} we show that it is not absolutely necessary to have non-decreasing abscissas: even if there are some decreasing abscissas 
one can  preserve the strong stability property by replacing the operator $L$ by an operator
$\tilde{L}$ that satisfies  \[\| e^{-\tau \tilde{L}} u^n \|  \leq \| u^n \|   \; \; \; \; \forall \; \tau \geq 0 , \]
whenever the difference of abscissas  $c_i-c_j$ is negative.
For hyperbolic partial differential equations, $\tilde{L}$ is simply the spatial discretization that is stable for the backwards in time version of the equation.
This approach is  employed  in the classical SSP literature, called downwinding. In that case, negative coefficients $\bij$ 
preserve the SSP property as long as the operator is replaced by a downwinded operator \cite{gottliebshu1998,gottlieb2001, ketcheson2011}.
%We studied this approach in \cite{SSPIFRK-downwind}.

\section{A review of explicit SSP two-step Runge--Kutta methods} \label{sec:background}

%SSP Runge--Kutta methods guarantee the strong stability (in any norm, semi-norm, or convex functional)
%of the numerical solution of any ODE  provided {\em only} that the forward Euler condition \eqref{FEstrongstability}  
%is satisfied under a time step restriction  \eqref{FEcond}.
%This  requirement  leads to severe restrictions on the allowable order of SSP methods,
%and the allowable time step $\Dt \le \sspcoef \DtFE.$ 
%These methods have been extensively studied, e.g., in 
%\cite{ferracina2004, ferracina2005,ferracina2008,SSPbook2011,gottliebshu1998,
%gottlieb2001, higueras2004a,higueras2005a, hundsdorfer2003, ketcheson2008, ketcheson2009,
%ketcheson2009a,  ketcheson2011, KubatkoKetcheson, ruuth2001}.
%In this section, we review some popular
%and efficient explicit SSP Runge--Kutta methods, and present the SSP coefficients of optimized methods
%of up to ten stages and fourth order. 

%Strong stability preserving methods were first developed by Shu \cite{shu1988, shu1988b} for use with total variation diminishing
%spatial discretizations. 

The order barrier of $p\leq 4$ on explicit SSP Runge--Kutta methods and the
need for higher order explicit methods in simulations led to the investigation of
the SSP properties of explicit two-step Runge--Kutta methods (TSRK) \cite{tsrk}, and later of 
explicit multi-step Runge--Kutta methods (MSRK) \cite{msrk}.
In \cite{tsrk} it was shown that explicit SSP TSRK  methods 
do not use values of stages from the previous steps, so that they can be written as:
\begin{subequations} \label{eq:mrktypeII}
\begin{align} 
y_1^n & =  u^n  \nonumber \\
y_i^n & =   d_{i} u^{n-1} + (1-d_{i} ) u^{n} + \Dt \hat{a}_{i1} F(u^{n-1}) + \Dt\sum_{j=1}^{i-1} a_{ij} F(y_j^n) \; \; \; \;  2 \leq i \leq s \\
u^{n+1} & =  \theta u^{n} + (1-\theta) u^{n-1}+  \Dt \hat{b}_{1} F(u^{n-1}) + \Dt\sum_{j=1}^s b_j F(y_j^n). \nonumber
\end{align} 
\end{subequations}
An extension of this to multiple steps was studied in \cite{msrk}. 

%Note that consistency requires that 
%\begin{subequations} \label{consistency}
%\begin{align} 
%\sum_{l=1}^k d_{il} & = 1  \quad 1\le i \le s,\\
%\sum_{l=1}^k \theta_l & = 1, 
%\end{align}
%\end{subequations}
We can write the method in the matrix form
\begin{align}
\label{spijkerform-compact}
\bw & = \mS\bx + \Dt \mT \vf
\end{align}
where $\mS$ and $\mT$ are the matrices defined by:
\begin{align} \label{msrk-spijker}
\mS & = \begin{pmatrix}1 \ 0   \\ \mD \\\btheta\transpose \end{pmatrix} \ \ \ \
\mT   = \begin{pmatrix}\mzero & \mzero & 0 \\ \mAh & \mA & \mzero \\ \bbh\transpose & \bb\transpose & 0 \end{pmatrix}
\end{align}
and $\bw$, $\bx$, and $\vf$ are defined by:
\begin{align} \label{wxf}
\bw = \begin{pmatrix} u^{n-1}\\y_1^n\\ \vdots \\y^n_s\\u^{n+1} \end{pmatrix} \ \ \ \
\bx = \begin{pmatrix}  u^{n-1} \\ u^{n} \end{pmatrix} \ \ \ \
\vf  = \begin{pmatrix} F\left(u^{n-1}\right)\\ F\left(y^n_1\right)\\ \vdots \\F\left(y^n_s\right)\\ F\left(u^{n+1}\right) \end{pmatrix}
\end{align}
The $\mD$ is a $2\times s$ matrix that contains the elements $\mD_{i,1} = d_i$ and $\mD_{i,2} = 1-d_i$, for $i=1,\dots, s$. The vector $\btheta$ contains the elements $\btheta_1 = \theta$ and $\btheta_2 = 1-\theta$. Note that $\mS \ve = \ve$, where $\ve$ is a vector ones of the proper length.

\subsection{Formulating the optimization problem} \label{sec:optimization}
In \cite{tsrk} a SSP optimization problem was defined for TSRK methods. The derivation of the optimization problem begins with the matrix form \eqref{spijkerform-compact} of the TSRK:
\[\bw  = \mS\bx + \Dt \mT \vf \]
We then add  $r \mT \bw$ to each side of the method
\[  \left( \mI +r \mT   \right)  \bw = 
\mS\bx  + r \mT \left( \bw + \frac{\Delta t}{r}  \vf \right) \]
 and rearrange to obtain
\[ \bw  =  \mR \bx + \mP \left( \bw + \frac{\Delta t}{r}  \vf \right) .\]
Where $\mP  =  r \left( \mI +r \mT   \right)^{-1}\mT $ and   $\mR  =  \left( \mI +r \mT   \right)^{-1}\mS=\left(\mI-\mP\right)\mS$. 
Note that the row sums of $[\mR ~ \mP]$ are each equal to one
\begin{eqnarray*}
\mR\ve+\mP\ve=\left(\mI-\mP\right)\mS\ve+\mP\ve=\mI\mS\ve-\mP\mS\ve+\mP\ve=\ve-\mP\ve+\mP\ve=\ve .
\end{eqnarray*}

It is easy to see that if $\mR$ and $\mP$ have no negative values, then  the method is a convex combination of forward Euler steps, and so is SSP with SSP coefficient defined by 
\begin{align*}
\sspcoef(\mS,\mT) & = \sup_{r}\left\{r : (\mI+r \mT)^{-1} \mbox{ exists and }  \mP \ge 0, \mR \ge 0
\right\}.
\end{align*}

This leads directly to the optimization problem: We wish to find coefficients $\mT$ and $\mS$ for explicit SSP two-step Runge--Kutta methods that 
maximize the value of $r$ while satisfying
\begin{subequations} \label{optimization}
\begin{align}
 \left( \mI +r \mT  \right)^{-1} \mS   \geq 0    \label{eq:constraints1}  \\
r  \left( \mI +r \mT  \right)^{-1} \mT \geq 0   \label{eq:constraints2}  \\
 \tau_\rho= 0 \; \; \; \mbox{for} \; \; \;  \rho=1, . . ., p   \label{eq:constraints3}. 
% c_1 \leq c_2 \leq \dots \leq c_s \leq 1 \label{eq:constraints4}
 \end{align}
 \end{subequations}
This requires that $\left( \mI +r \mT  \right)^{-1}$ exists. The inequalities \eqref{eq:constraints1} and \eqref{eq:constraints2} are understood componentwise. 
The $\tau_\rho$ in \eqref{eq:constraints3} are the order conditions listed in Appendix \ref{orderconditions}.

\subsection{Optimized SSP two-step Runge--Kutta methods}
In \cite{tsrk} a SSP optimization problem was defined for the TSRK methods, and optimized explicit TSRK methods of up to eighth order were found numerically.  The SSP coefficients $\sspcoef$ of these methods up to seventh order are in Table \ref{tab:SSPcoef}. Of course, a method with more stages will have a higher SSP coefficient, but also require more computations. To account for this we define the {\em effective SSP coefficient}  $\ceff = \frac{1}{s} \sspcoef$, which is the SSP coefficient normalized by the number of stages, in Table \ref{tab:effSSPcoef}. An eleven stage eighth order method has an SSP cefficient $\sspcoef=0.341$ and an effective SSP Coefficient $\ceff=0.031$.

These methods break the fourth-order barrier of (one-step) SSP Runge--Kutta methods and have significantly larger SSP coefficients than the corresponding order multi-step methods. Numerical tests in \cite{tsrk} showed that these high order SSP two-step Runge-Kutta methods are beneficial for preserving the order
and strong stability properties when used with high order spatial discretizations 
for the  time integration of a variety of hyperbolic PDEs.
 It was proved in \cite{tsrk} that explicit SSP TSRK methods have order at most eight.  
 In \cite{msrk} it was shown that adding more steps overcomes this order barrier as well.

\begin{table}[h]
\centering
\setlength\tabcolsep{4pt}
%\begin{minipage}{0.48\textwidth}
\centering
\begin{tabular}{|c|llllll|} \hline
  \diagbox{s}{p}  &  2 & 3 & 4 & 5 & 6 & 7\\
 \hline
    1  &  -           &  -           &   -		&   -		&   -		&   -\\   
    2  &  1.4142 &  0.7320 &   -		&   - 		&   - 		&   -\\    
    3  &  2.4495 &  1.6506 &   0.8588  &   -		&   - 		&   -\\
    4  &  3.4641 &  2.3027 &   1.5926  &   0.8542 &  - 		&   -\\
    5  &  4.4721 &  2.9879 &   2.3605 &   1.6481 &   - 		&   -\\
    6  &  5.4772 &  3.7768 &   3.0559 &   2.3093 &   0.5957 &   -\\
    7  &  6.4807 &  4.4836 &   3.7405 &   2.9278 &   1.2719 &   -\\
    8  &  7.4833 &  5.2227 &   4.4921 &   3.5794 &   1.9384 &  0.5666 \\
    9  &  8.4853  &  6.0498 &   5.2705 &   3.9415 &   2.5826 & 1.1199  \\
   10 &  9.4868  &  6.8274 &   6.1039 &   4.2544 &   3.1992 & 1.7857 \\
\hline
\end{tabular}
\caption{SSP coefficients of the optimized eSSP-TSRK(s,p) methods \cite{tsrk}.}
\label{tab:SSPcoef} 
 \end{table}
%\end{minipage}%
%\hfill
%\begin{minipage}{0.48\textwidth}
\begin{table}[h]
\centering
\setlength\tabcolsep{4pt}
\centering
\begin{tabular}{|c|llllll|} \hline
  \diagbox{s}{p}  &  2 & 3 & 4 & 5 & 6 & 7\\
 \hline
    1  &  -           &  -           &   -		&   -		& - 	       &   -\\   
    2  &  0.7071 &  0.3660 &   -		&  -		& - 	       &   -\\    
    3  &  0.8165 &  0.5502 &   0.2863 &   - 		& - 	       &   -\\
    4  &  0.8660 &  0.5757 &   0.3982	&  0.2135 & - 	       &   -\\
    5  &  0.8944 &  0.5976 &   0.4721 &   0.3296 & -	       &   -\\
    6  &  0.9129 &  0.6295 &   0.5093 &   0.3849 & 0.0993 &   -\\
    7  &  0.9258 &  0.6405 &   0.5343 &   0.4183 & 0.1817 &   -\\
    8  &  0.9354 &  0.6528 &   0.5615 &   0.4474 & 0.2423 &  0.0708 \\
    9  &  0.9428 &  0.6722 &   0.5856 &   0.4379 & 0.2869 &  0.1244\\
   10 &  0.9487 &  0.6827 &   0.6104 &   0.4254 & 0.3199 &  0.1786 \\
\hline
\end{tabular}
\caption{Effective SSP coefficients of the optimized eSSP-TSRK(s,p) methods \cite{tsrk}.}
 \label{tab:effSSPcoef} 
%\end{minipage}
\end{table}

\section{Explicit SSP two-step Runge--Kutta schemes for use with integrating factor methods}
 \label{sec:SSPIF} 

In this section we consider, as we did in \cite{SSPIFRK-SINUM,SSPIFRK-downwind}, a problem of the form
\eqref{IFproblemform}
%\begin{eqnarray}
%u_t = Lu + N(u)  
%\end{eqnarray}
with a linear constant coefficient component $L u $ 
that satisfies the strong stability condition \eqref{linearFEcond} and a nonlinear component $N(u)$ that satisfies the strong stability condition \eqref{nonlinearFEcond}.
When $ \tDtFE <<  \DtFE$  using an explicit SSP Runge--Kutta method
will result in severe constraints on the allowable time-step that is not alleviated by using 
 an implicit  or an  implicit-explicit (IMEX) SSP Runge--Kutta method  \cite{lenferink1991,ketcheson2009,ferracina2008,SSPbook2011,SSPIMEX}.
This is in contrast to the case with linear $L_2$ stability, where using an implicit or an IMEX method
may completely alleviate the time-step restriction coming from the stiff component.

To alleviate the restriction on the allowable time-step we can solve the linear part exactly by an integrating factor approach as in \eqref{IF_ODE}.
%\begin{eqnarray*}
%e^{-L t } u_t - e^{-L t }  Lu = e^{-L t } N(u)  \\
%\left( e^{-L t } u \right)_t = e^{-L t } N(u). \\
%w_t  = e^{-L t } N(e^{L t } w)   = G(w).
%\end{eqnarray*}
This new ODE can be evolved forward in time using standard methods.
As described in Section \ref{sec:Review-IFRK}, in \cite{SSPIFRK-SINUM} we considered stepping the transformed problem \eqref{IF_ODE} 
forward using an explicit SSP Runge--Kutta method of the form \eqref{rkSO} and found that the result was SSP provided that 
the method had non-decreasing abscissas. In the following, we consider extending the SSP integrating factor approach by 
using multi-step and two-step Runge--Kutta integrating factor methods to obtain order $p\geq4$

\subsection{SSP integrating factor linear multi-step methods} \label{sec:multistepif}
%\noindent{\bf SSP integrating factor linear multistep methods:}
It can be  easily shown that if we use an explicit SSP multi-step method
\begin{eqnarray} \label{SSPms}
u^{n+1} & = &  \sum_{l=1}^{k}  \left( \alpha_l u^{n-k+l} + \dt \beta_l  F(u^{n-k+l}) \right),
\end{eqnarray}
which satisfies 
 \[   u^{n+1}  \leq \max_{l=1, . . ., k} \left\{ \|u^{n-k+l} \| \right\}. \]
with SSP coefficient $\sspcoef= \max_{l} \frac{\alpha_l}{\beta_l}$, 
 the result is SSP as well, with the same SSP coefficient:

\begin{theorem} \label{thm:SSPIF-MS}
Given a  linear operator $L$ that satisfies   \eqref{EXPcondition}
and a (possibly nonlinear) operator $N(u)$ that satisfies \eqref{nonlinearFEcond}  
for some value of ${\Delta t}_{FE} > 0 $, and a $k$-step explicit SSP integrating factor multi-step method of the form
\begin{eqnarray} \label{msIFSO}
 u^{n+1} & = &  \sum_{l=1}^{k} e^{L (k -l +1 )} \left( \alpha_l u^{n-k+l} + \dt \beta_l  N(u^{n-k+l}) \right),
\end{eqnarray}
 then the numerical solution satisfies
\begin{equation}
\|u^{n+1}\| \leq  \max_{l=1, . . ., k} \left\{ \|u^{n-k+l} \| \right\}.  \; \; \; \forall \dt \leq \sspcoef \DtFE.
\end{equation}
where  $\sspcoef= \max_{l} \frac{\alpha_l}{\beta_l}$. %\underset{1 \leq j \leq n}{\max}
\end{theorem}
 Notice that for multi-step methods the notion of SSP includes all the previous steps considered.
\begin{proof}
We write the method in the form
\begin{eqnarray*}
\| u^{n+1}\|  & = & \|  \sum_{l=1}^{k} e^{L (k -l +1 )} \left( \alpha_l u^{n-k+l} + \dt \beta_l  N(u^{n-k+l}) \right) \| \\
& \leq &   \sum_{l=1}^{k} \| e^{L (k -l +1 )} \left( \alpha_l u^{n-k+l} + \dt \beta_l  N(u^{n-k+l}) \right) \| \\
& \leq &  \sum_{l=1}^{k} \alpha_l  \| e^{L (k -l +1 ) } \left(  u^{n-k+l} + \dt \frac{\beta_l}{\alpha_l}  N(u^{n-k+l}) \right) \| \\
& \leq &  \sum_{l=1}^{k} \alpha_l \|u^{n-k+l} \|.
\end{eqnarray*}
the last inequality holds provided that  $\dt \frac{\beta_l}{\alpha_l}  \leq \DtFE$. 
Using the fact that by consistency we have $\sum_{l=1}^{k} \alpha_l =1$, we can conclude that 
\[ \| u^{n+1}\|  \leq \max_{l=1, . . ., k} \left\{ \|u^{n-k+l}\| \right\}. \]
\end{proof}
However, SSP multi-step methods require many steps for high order and have small SSP coefficients as shown in \cite{SSPbook2011}.
Now that we see that building explicit SSP multi-step integrating factor methods is trivial, 
we are ready to approach the main question of this paper: is it possible to extend our results in
\cite{SSPIFRK-SINUM} to two-step Runge--Kutta methods?

%\noindent{\bf SSP integrating factor two-step Runge--Kutta methods:}
\subsection{SSP integrating factor two-step Runge--Kutta methods}  %\label{sec:twostepif}
Any explicit SSP TSRK  method of the form \eqref{eq:mrktypeII} that is 
SSP with SSP coefficient $\sspcoef=r$  can be re-written in the matrix-vector form \eqref{spijkerform-compact},
which can also be represented as
\begin{subequations} \label{eq:SOtsrk}
\begin{align}
y_1 & =  u^n  \nonumber \\
y_i & =  v_{i,1} u^{n-1} + v_{i,2} u^{n} +  \hat{\alpha}_i \left( u^{n-1} + \frac{\dt}{r}  F(u^{n-1} ) \right)\\
& + \sum_{j=1}^{i-1} \aij  \left(  y_j + \frac{\dt}{r}  F(y_j ) \right)   \; \; \; \;  2 \leq i \leq s+1 \nonumber \\
u^{n+1} & =  y_{s+1} , \nonumber
\end{align} 
\end{subequations}
 where $0 \leq v_i \leq 1$, $\hat{\alpha}_i \geq 0$, and $\aij \geq 0$.
Here, the consistency condition $\mR\ve+\mP\ve=\ve$ can be written as 
$v_{i,1}  + v_{i,2}  +  \hat{\alpha}_i + \sum_{j=1}^{i-1} \aij   = 1$ for each row $i$.

We note that the time-levels for each stage, also called the abscissas, \[ \vc = \left( c_1, c_2, c_3, . . . , c_s, c_{s+1} \right) = \left(0, c_2, c_3, . . . , c_s, 1\right)\] can be computed (in the matrix form) by $\mD$, $\mA$, and $\mAh$ where $\vc = \mAh\ve+\mA\ve -\mD_{i,1}$. 
%which we can put into matrices $\tilde{\mD}$ and $\tilde{\mA}$ defined by:
%\begin{align} 
%\tilde{\mD} & = \begin{pmatrix}\mI_{(k-1) \times (k-1)} \ \mzero_{1 \times (k-1)} \\ \mD \\   \end{pmatrix} \ \ \ \
%\tilde{\mA}   = \begin{pmatrix}\mzero & \mzero  \\ \mAh & \mA  \\  \end{pmatrix} \ \ \ \
%%\tilde{\bb} =  \begin{pmatrix} \bbh & \bb  \end{pmatrix}, 
%\end{align}
%We let $\vl$ be the vector $\left(k-1, k-2, ..., 1,0 \right)^T$ and compute $\vc = \tilde{\mA} \ve -\tilde{\mD} \vl$.
%The components of this vector.
%\[ \vc = \left( c_1, c_2, c_3, . . . , c_s, c_{s+1} \right) = \left(0, c_2, c_3, . . . , c_s, 1\right)\]
% are called the abscissas.
%abscissa corresponding to $u^{n-1}$ is $-1$, the one corresponding to
The abscissa corresponding to $u^n$ is $0$, the ones corresponding 
to all the internal stages $j=2, . . ., s$ are $c_j$ and the one corresponding to the final stage $u^{n+1}$ is $1$.

Applying \eqref{eq:SOtsrk} to the transformed equation becomes
\begin{subequations} \label{IFmsrk}
\begin{align}
y_1 & =  u^n \nonumber \\
y_i & =  v_{i,1} e^{(c_i +1)\dt } u^{n-1} + v_{i,2}u^{n} +  \hat{\alpha}_i  e^{(c_i +1)\dt } \left( u^{n-1} + \frac{\dt}{r}  N(u^{n-1} ) \right) \\
&+\sum_{j=1}^{i-1} \aij e^{(c_i -c_j )\dt }  \left(  y_j + \frac{\dt}{r}  N(y_j ) \right)   \; \; \; \;  2 \leq i \leq s+1 \nonumber  \\
u^{n+1} & =  y_{s+1} \nonumber \end{align} 
\end{subequations}

The following theorem tells us what conditions an explicit SSP TSRK method \eqref{eq:SOtsrk} 
has to satisfy to preserve the strong stability of a numerical solution when
coupled with an integrating factor approach as in \eqref{IFmsrk}.

\begin{theorem} \label{thm:SSPIF-TSRK}
Given a linear operator $L$ that satisfies   \eqref{EXPcondition}
and a (possibly nonlinear) operator $N(u)$ that satisfies \eqref{nonlinearFEcond}  
for some value of ${\Delta t}_{FE} > 0 $, and an explicit SSP  two-step  Runge--Kutta integrating factor method of the form
\eqref{IFmsrk}  based on the SSP method \eqref{eq:SOtsrk}
with non-decreasing abscissas  $0=c_1 \leq c_2 \leq c_3 ....... \leq c_{s} \leq c_{s+1} = 1$,
then the numerical solution satisfies
\begin{equation}
\|u^{n+1}\| \leq \max \left\{ \|u^{n-1}\| , \|u^n\| \right\}  \; \; \; \forall \dt \leq \sspcoef \DtFE.
\end{equation}
where  the  SSP coefficient $\sspcoef=r$.
\end{theorem}
\begin{proof}
Consider each stage of \eqref{IFmsrk} 
\begin{eqnarray*}
\| y_i \| & = & \| v_{i,1} e^{(c_i +1)\dt } u^{n-1} + v_{i,2} u^{n} +  \hat{\alpha}_i  e^{(c_i +1)\dt } \left( u^{n-1} + 
\frac{\dt}{r}  N(u^{n-1} ) \right)  \\
&+& \sum_{j=1}^{i-1} \aij e^{(c_i -c_j )\dt }  \left(  y_j + \frac{\dt}{r}  N(y_j ) \right)   \| \\
& \leq  &  v_{i,1}  \| e^{(c_i +1)\dt L } u^{n-1} \| + v_{i,2} \| u^{n} \| +  \hat{\alpha}_i  \| e^{(c_i +1)\dt } \left( u^{n-1} +  
 \frac{\dt}{r}  N(u^{n-1} ) \right) \| \\
 &+& \sum_{j=1}^{i-1} \aij \| e^{(c_i -c_j )\dt L }  \left(  y_j + \frac{\dt}{r}  N(y_j ) \right) \|   
\end{eqnarray*}
Now noting that $c_i - c_j \geq 0$ and applying the result in Corollary 1, we obtain
\begin{eqnarray*}
\| y_i \|  & \leq  & v_{i,1} \| u^{n-1} \|  +   v_{i,2} \| u^{n} \|  + \hat{\alpha}_i \|u^{n-1} \|  + \sum_{j=1}^{i-1} \aij \|u^n \| \\
& \leq  & \max \{\| u^n\| , \| u^{n-1} \| \} 
\end{eqnarray*}
for any $ \dt \leq \sspcoef {\Delta t}_{FE} $.
%\begin{eqnarray*}
%\| u^{(i)} \|& = & \left\| \sum_{j=0}^{i-1} e^{L (c_i-c_j) \dt }  \left( \aij u^{(j)} + \dt \bij N(u^{(j)}) \right) \right\| \\
%& \leq  & \sum_{j=0}^{i-1} \left\|  e^{L (c_i-c_j) \dt }  \left( \aij u^{(j)} + \dt \bij N(u^{(j)}) \right)\right\| \\
%& \leq  &  \sum_{j=0}^{i-1} \aij \left\|  e^{L (c_i-c_j) \dt }  \left( u^{(j)} + \dt \frac{\bij}{\aij} N(u^{(j)}) \right) \right\| 
%\end{eqnarray*}
%where the last inequality follows from Corollary 1, as long as $c_i - c_j \geq 0$ and 
%$\dt \frac{\bij}{\aij}  \leq \DtFE$. This establishes the result of the theorem.
%Furthermore, this proof ensures that these methods have internal stage strong stability as well,  i.e.
%$\| u^{(i+1)} \| \leq \| u^{(i)} \| $ at each stage $i$ of the time-stepping, under the same time-step restriction.
\end{proof}
This theorem shows us that to find optimal SSP two-step Runge--Kutta methods that are suitable for use with the 
integrating factor approach, we must augment the optimization problem with additional conditions reflecting the need for non-decreasing abscissas.

To see how this theorem matters in practice, consider, once again, the motivating example in Subsection \ref{motivating}
 where we solve Equation \eqref{BurgersAdvection}
using an integrating factor approach. In this case, we use $a=2$ so $L u \approx -2 u_x$. Otherwise, everything is the same
as in  Subsection \ref{motivating}. We evolve the transformed problem forward 25 time-steps for various time-steps
using the  eSSP-TSRK$(10,5)$ method in \cite{tsrk} and  a new eSSP-TSRK$^+(10,5)$ method that has
non-decreasing abscissas. The $\log_{10}$ of the maximal rise in total variation at each stage vs. the value $\lambda = \frac{\dt}{\dx}$
is shown in Figure \ref{fig:motivating2}. We observe that the new method preserves the TVD behavior of the 
spatial discretization up to a large $\lambda$, while the eSSP-TSRK$(10,5)$ method in \cite{tsrk} does not preserve the TVD 
property of the spatial discretization.

\begin{figure}
\centering
 \includegraphics[scale=.375]{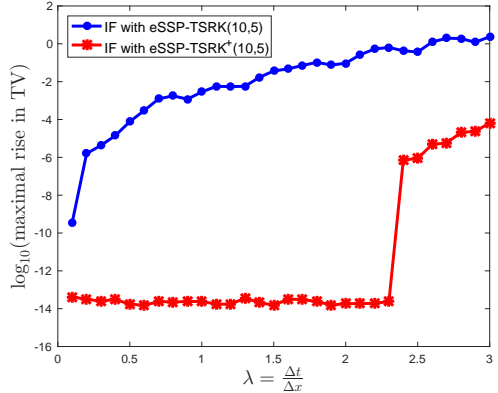} 
    \caption{Total variation behavior of the evolution over 25 time-steps for the motivating example in  Subsection \ref{motivating},
    using the  integrating factor methods based on eSSP-TSRK$^+(10,5)$ (red stars) and eSSP-TSRK$(10,5)$ (blue dots).
}
\label{fig:motivating2}
\end{figure}

\section{Optimal and optimized methods with non-decreasing abscissas\label{sec:optimal}}

We augmented the optimization problem  from \cite{tsrk} described in Subsection \ref{sec:optimization}  by 
imposing the non-decreasing abscissa constraint 
\begin{eqnarray} \label{nnabscissa}
0=c1 \leq c_2 \leq c3 \leq. . . . \leq c_{s-1} \leq c_s \leq c_{s+1} =1
\end{eqnarray}
 in addition  to  \eqref{optimization}.
We then implemented in {\sc Matlab} (as in \cite{ketcheson2008, ketcheson2009a, ketcheson2009}), 
and used to find optimized explicit eSSP-TSRK$^+$ methods of $s \leq 10$ stages and order $p\leq 7$,
as well as an eighth order method  with $s=11$. 
The new methods have non-decreasing abscissas, and are denoted eSSP-TSRK$^+(s,p)$. 
As we saw above, these  methods preserve the SSP properties of a transformed problem \eqref{IF_ODE}
and so can be used as a basis for explicit SSP integrating factor two-step Runge--Kutta (SSPIF-TSRK) methods.

\begin{figure}[ht]
\includegraphics[scale=.3]{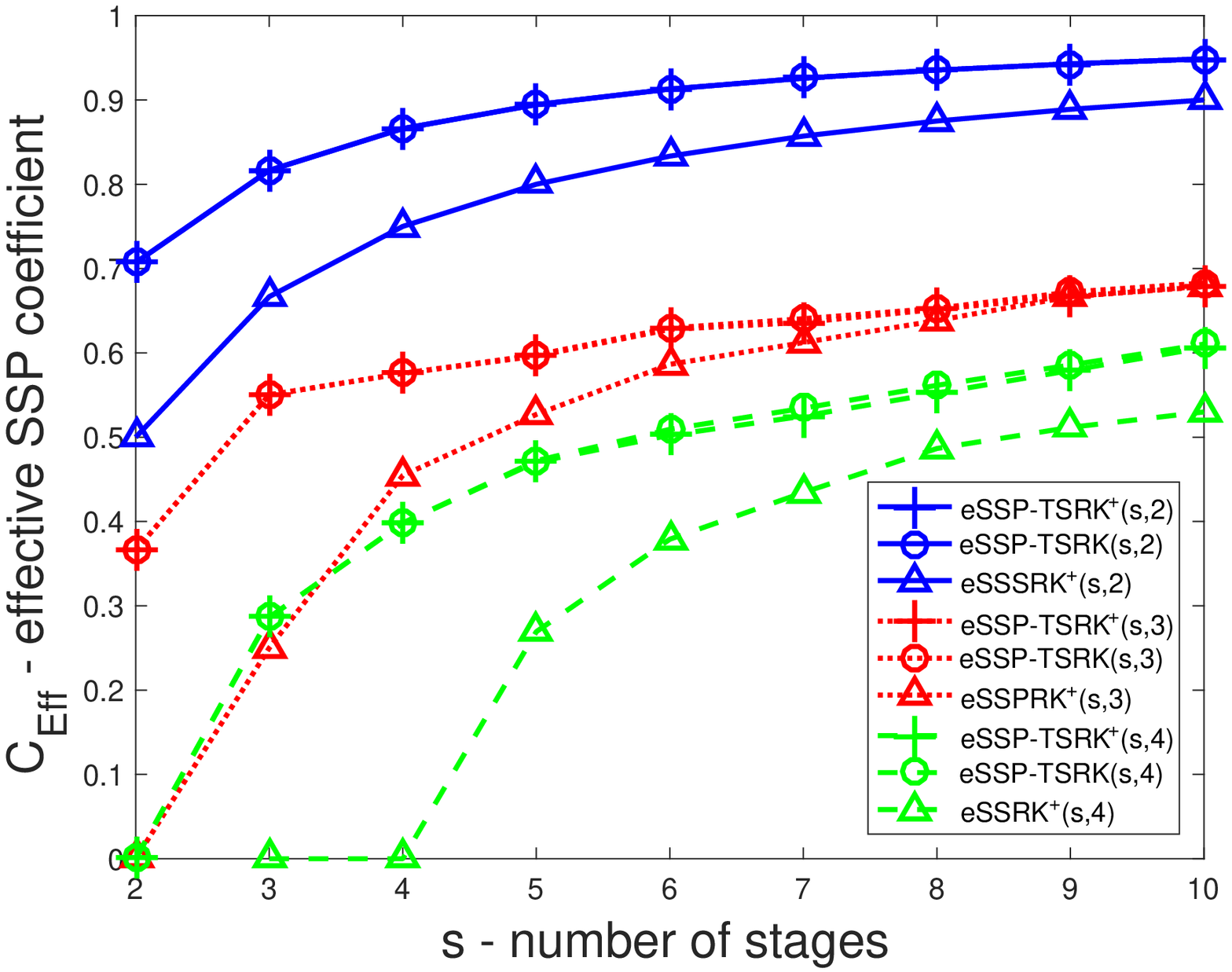} 
\includegraphics[scale=.3]{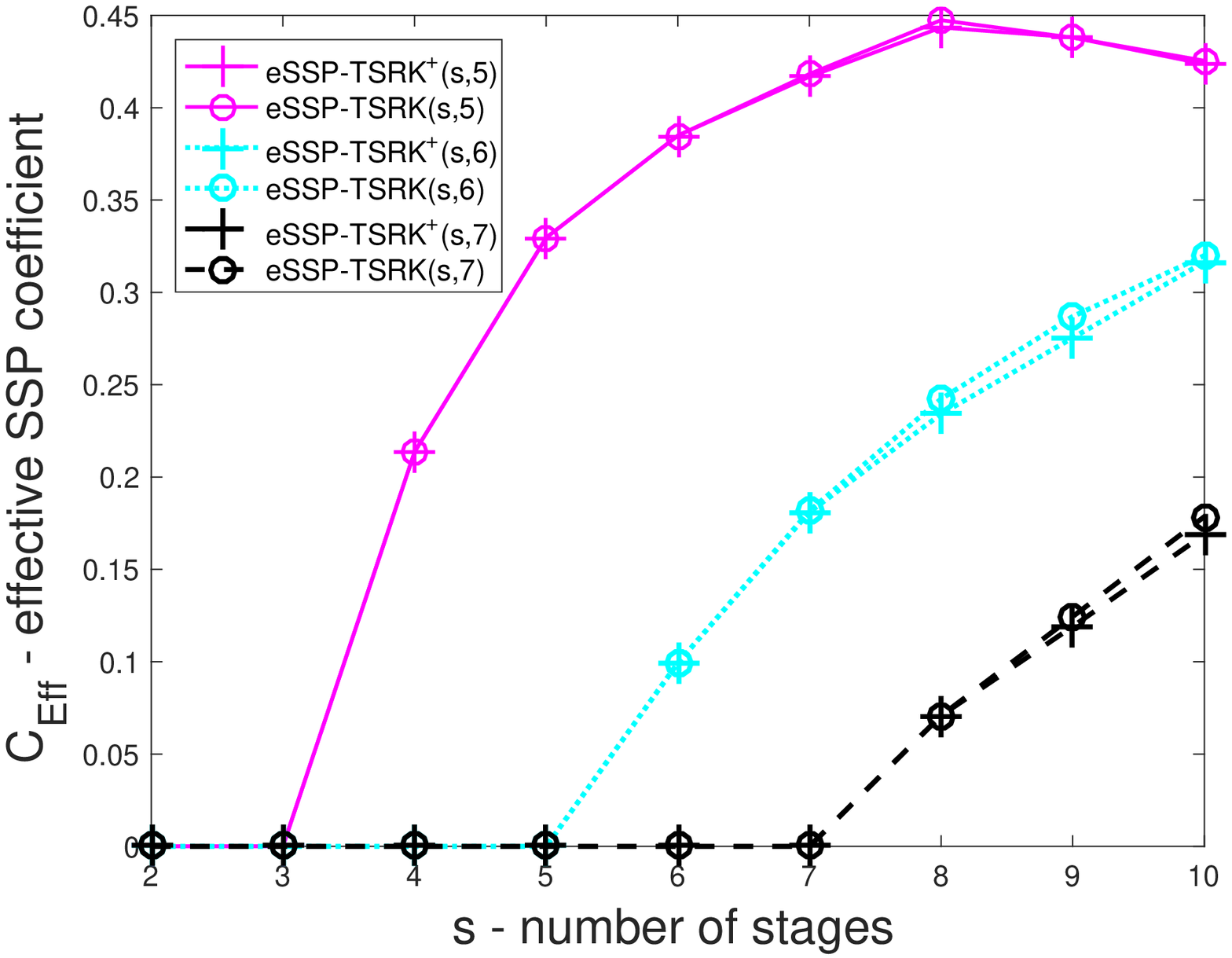}\\ \vspace{-.2in}
\caption{The  effective SSP coefficient  for methods of  order 
 $p=2$ (blue), $p=3$ (red) and $p=4$ (green) on left, and of
 order  $p=5$ (magenta), $p=6$ (cyan) and $p=7$ (black) on right.
 The circles indicate the SSP coefficient of the optimized  eSSP-TSRK methods in \cite{tsrk}
 while the + signs are the SSP coefficients of the optimized  eSSP-TSRK$^+$ methods.
The triangles show the SSP coefficients of the corresponding eSSPRK$^+$ methods (where relevant).}
 \label{fig:SSPcoeff}
\end{figure}

The SSP coefficients and effective SSP coefficients of the optimized methods with non-decreasing abscissas (eSSP-TSRK$^+$) are listed in Tables \ref{tab:SSPcoefIF} and \ref{tab:effSSPcoefIF}. We boldface  the coefficients of the eSSP-TSRK$^+$ methods that are the same as the optimized eSSP-TSRK methods (previously found in \cite{tsrk,msrk})  shown in Tables \ref{tab:SSPcoef} and \ref{tab:effSSPcoef}. 
In Figure \ref{fig:SSPcoeff} we show the effective SSP coefficients of the eSSP-TSRK$^+$ methods and compare them to those of the corresponding eSSP-TSRK methods. For orders $p\leq 4$ we also show the effective SSP coefficients of the eSSPRK$^+$, for comparison. These tables and graphs include all the second order $p=2$ methods, which already have non-decreasing coefficients.
  
 The third order  eSSP-TSRK$^+(s,3)$ methods with $s=2,3,4$ also  have the same effective SSP coefficients as the eSSP-TSRK$(s,3)$ methods in \cite{tsrk}. Although the eSSP-TSRK$(3,3)$ method in \cite{tsrk} has decreasing coefficients, we were able to find an eSSP-TSRK$^+(3,3)$ non-decreasing coefficient method with the same SSP coefficient. The eSSP-TSRK$^+$(s,3) methods with $s>4$ have smaller SSP coefficients than the corresponding eSSP-TSRK methods, but the loss in size of effective SSP coefficient is under $1$\%.   These methods have a significant advantage over the eSSPRK$^+$ methods, but this advantage disappears as we allow a larger number of stages $s$, as can be seen in Figure \ref{fig:SSPcoeff}.

The fourth order eSSP-TSRK$^+(s,4)$  methods with $s=3,4$ also  have the same SSP coefficients as the
corresponding eSSP-TSRK  in \cite{tsrk}.
These two methods are significant because allowing an additional step makes possible fourth order SSP methods with $s=3,4$, which cannot be achieved for  eSSPRK methods.
Once we look at more stages, $s>4$, we see that the eSSP-TSRK$^+(s,4)$ methods suffer some loss
over the eSSP-TSRK$(s,4)$ in the magnitude of the effective SSP coefficient, but it is still under $2$\%,  
so the differences are not  obvious in Figure \ref{fig:SSPcoeff}. Once again, these eSSP-TSRK$^+(s,4)$  
methods offer a significantly larger effective SSP coefficient than the corresponding eSSPRK$^+$ methods. 

The fifth order eSSP-TSRK$^+(s,5)$  methods with $s=4,5,6,9$ also have the same SSP coefficients as the 
corresponding eSSP-TSRK methods  in \cite{tsrk}. We note that our optimization code found a slightly better 
eSSP-TSRK$^+(9,5)$ than the eSSP-TSRK$(9,5)$ previously found in the literature.
For $s=7,8,10$ the  eSSP-TSRK$^+(s,5)$ cannot match the SSP coefficients of the  eSSP-TSRK$(s,5)$ methods in \cite{tsrk}. The loss that we see  in the effective SSP coefficients is, once again, under $1$\%, so the difference cannot be seen in Figure \ref{fig:SSPcoeff}.  It is important to note that there are no explicit SSP Runge--Kutta methods of fifth order or above, so the eSSP-TSRK methods are particularly useful if one needs to go to higher order.

The sixth order eSSP-TSRK$^+(s,6)$ methods we found all have smaller SSP coefficients than the 
eSSP-TSRK$(s,6)$ in \cite{tsrk}.  For $s\leq 7$ the loss is under $0.5$\% which cannot be seen in Figure  \ref{fig:SSPcoeff}. However, once we get to $s=8,9,10$ we see more significant losses 
in the size of the SSP coefficients, however these are still below $5$\%. It is interesting to notice that the 
explicit eSSP-TSRK$^+(10,6)$ method has an effective SSP coefficient $\ceff \approx 0.31$ that is about half that of the eSSPRK$(10,4)$ (which is $\ceff =0.6$) and about $60$\% of that of the eSSPRK$^+(10,4)$
(which is $\ceff \approx 0.5299$).

The  eSSP-TSRK$^+(8,7)$   method we found has the same SSP coefficient as the eSSP-TSRK$(8,7)$ in \cite{tsrk}, while the eSSP-TSRK$^+(s,7)$ methods with $s=9,10$ have a loss in the effective SSP coefficient of $4.4$\% and $5.55$\% compared with the corresponding  eSSP-TSRK methods in \cite{tsrk}.
 Although it is possible to find eighth order methods, these require more than ten stages. Using our optimization code, we found an eleven stage eighth order method with non-decreasing abscissas  eSSP-TSRK$^+$(11,8)
with an effective SSP coefficient $\ceff \approx 0.0249$. The SSP coefficient of the eSSP-TSRK$^+(11,8)$ methods is significantly smaller (by close to 20\%) than the eleven stage eighth order method eSSP-TSRK$(11,8)$  in \cite{tsrk}  which has $\ceff \approx 0.031$.
  
In general, we see that for lower number of stages, the SSP coefficients of the methods with non-decreasing abscissas are no smaller than the methods with decreasing abscissas. Perhaps this reflects the fact that the solution space for methods with $s$ close to $p$ is already limited, and so the additional constraint of non-decreasing abscissas poses no major restriction. However, it is interesting to notice that despite the constrained solution space when $s$ is close to $p$, in every case where an eSSP-TSRK$(s,p)$ method was found 
it was also possible to find an eSSP-TSRK$^+(s,p)$ method.
  
% The SSP coefficients of this family of methods are  compared to those of the optimized  eSSPRK methods with no constraint
% on the abscissas in Figure
% \ref{fig:SSPcoeff}, where the circles indicate the SSP coefficient of the optimized  explicit SSPRK methods 
%  while the lines are the SSP coefficients of the optimized  explicit SSPRK$^+$ methods.}
% 
%The Tables \ref{tab:SSPcoefIF} and \ref{tab:effSSPcoefIF} list SSP coefficients and effective SSP coefficients of the optimized eSSPTSRK$^+$ methods. 

\begin{table}[h]%[t!]
\centering
\setlength\tabcolsep{4pt}
%\begin{minipage}{0.48\textwidth}
\centering
\begin{tabular}{|c|llllll|} \hline
  \diagbox{s}{p}  &  2 & 3 & 4 & 5 & 6 & 7\\
 \hline
    1  &  -           		&  -           &   -			&   -		&   -		&   -\\   
    2  &  {\bf 1.4142} &  {\bf 0.7320} &   -		&   - 			&   - 		&   -\\    
    3  &  {\bf 2.4495} &  {\bf 1.6506}	 & {\bf 0.8588}  &   -		&   - 		&   -\\
    4  &  {\bf 3.4641} &  {\bf 2.3027} & {\bf 1.5926}  & {\bf 0.8542} &  - 		&   -\\
    5  &  {\bf 4.4721} &  2.9807 	&   2.3523 	&   {\bf 1.6481} 	&   - 		&   -\\
    6  &  {\bf 5.4772} &  3.7672 	&   3.0140		&   {\bf 2.3093} 	&   0.5958 &   -\\
    7  &  {\bf 6.4807} &  4.4533 	&   3.6751		&   2.9173 	&   1.2671 &   -\\
    8  &  {\bf 7.4833} &  5.2134 	&   4.4178 	&   3.5477 	&   1.8728 & {\bf 0.5666}\\
    9  &  {\bf 8.4853} &  6.0012 	&   5.2120 	&   {\bf 3.9426} &   2.4784 & 1.0715  \\
   10 &  {\bf 9.4868} &  6.7916 	&   6.0626 	&   4.2362 	&   3.1646 & 1.6892 \\
\hline
\end{tabular}
\caption{SSP coefficients of the optimized eSSP-TSRK$^+$(s,p) methods.}
\label{tab:SSPcoefIF} 
%\end{minipage}%
\end{table}
%\hfill
%\begin{minipage}{0.48\textwidth}
  \begin{table}[h!]%[t!]
\centering
\setlength\tabcolsep{4pt}
\centering
\begin{tabular}{|c|llllll|} \hline
  \diagbox{s}{p}  &  2 & 3 & 4 & 5 & 6 & 7\\
 \hline
    1  &  -           		&  -           		&   -		&   -		& - 	       &   -\\   
    2  &  {\bf 0.7071} &  {\bf 0.3660} &   -		&  -		& - 	       &   -\\    
    3  &  {\bf 0.8165} & {\bf 0.5502}	 &   {\bf 0.2863} &   - 		& - 	       &   -\\
    4  &  {\bf 0.8660} &  {\bf 0.5757} &   {\bf 0.3982}&  {\bf 0.2135} & - 	       &   -\\
    5  &  {\bf 0.8944} &  0.5961 	&   0.4705 	&   {\bf 0.3296} & -	       &   -\\
    6  &  {\bf 0.9129} &  0.6279 	&   0.5023 	&   {\bf 0.3849} & 0.0993 &   -\\
    7  &  {\bf 0.9258} &  0.6362 	&   0.5250 	&   0.4168 	& 0.1810 &   -\\
    8  &  {\bf 0.9354} &  0.6517 	&   0.5522 	&   0.4435 	& 0.2341 &  {\bf 0.0708}\\
    9  &  {\bf 0.9428}  &  0.6668 	&   0.5791 	&   {\bf 0.4381} & 0.2754 &  0.1191\\
   10 &  {\bf 0.9487}  &  0.6792 	&   0.6063 	&   0.4236 	& 0.3165 &  0.1689 \\
\hline
\end{tabular}
\caption{Effective SSP coefficients of the optimized eSSP-TSRK$^+$(s,p) methods.}
 \label{tab:effSSPcoefIF} 
%\end{minipage}
\end{table}

\section{Numerical Results\label{sec:test}}

In this section we demonstrate the performance of 
 the SSP integrating factor  two-step Runge--Kutta (SSPIF-TSRK) 
methods based on the  eSSP-TSRK$^+$ methods  reported in Section \ref{sec:optimal}. 
The SSP coefficients of these methods are presented in Table  \ref{tab:SSPcoefIF};
 the coefficients of these methods can be downloaded from \cite{SSPIF-TSRKgithub}.
 
Our tests focus on convergence and  the SSP properties of the methods.
In Subsection  \ref{sec:Num_vdp}  we  verify convergence  at the design order
for a selection of these methods on the van der Pol problem.
In Subsections \ref{sec:Num_linear} and \ref{sec:Num_nonlinear} we study the behavior of these methods in terms of their allowable 
TVD time-step on  linear and nonlinear problems with spatial discretizations that are TVD. 
The simple TVD test in these  sections  have been used extensively because they  demonstrate the behavior of the
SSP time-step, and provide evidence that there are some cases in which the SSP property is necessary to preserve the TVD 
property.

\subsection{Example 1: Convergence study}\label{sec:Num_vdp}  
Consider the van der Pol problem, a nonlinear system of ODEs of the form
\begin{eqnarray*}
u_1' &=& u_2 \\ 
u_2' &=&  (-u_1 + (1-u_1^2) u_2).
\end{eqnarray*}
The problem can be written as
$\vu_t =L \vu + N(\vu)$ where  $\vu=(u_1;u_2)$ and 
 \[L =  \left( \begin{array}{rr}
0 & 1 \\
-1 & 0\\
\end{array} \right), \; \; \; 
N(\vu) = \left(\begin{array}{c}
0 \\ (1-u_1^2) u_2
\end{array} \right).
 \]
We initialize the problem with $\vu_0 = (2;  0)$. Using a variety of SSPIF-TSRK methods, we run  the problem to 
final time $T_{final} =  2$, with $\Delta t = 0.01, 0.02, 0.04, 0.05, 0.08, 0.10$. The initial values and the exact solution 
(for error calculation)  was calculated by {\sc Matlab}'s ODE45 routine with tolerances set to
 {\tt AbsTol=}$2\times10^{-14}$ and {\tt  RelTol=}$2\times10^{-14}$. 
We tested  the SSPIF-TSRK methods based on the eSSP-TSRK$^+$ methods whose SSP coefficients are
presented in Table \ref{tab:SSPcoefIF} above. We  calculated  the orders by finding the slopes of the $\log_{10}(errors)$ using
 {\sc Matlab}'s {\tt polyfit} function.
The results, shown in Figure \ref{fig:VDPconv} for several selected methods, 
validate that the methods exhibit the expected order of convergence.
\begin{figure}[t]
\begin{center}
\includegraphics[scale=.375]{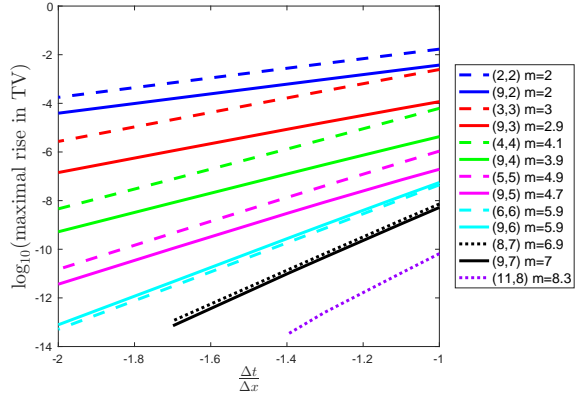}
\end{center} 
\caption{Convergence plot for Example 1. 
%The x-axis is $\log_{10}$ of  the time-step, while the y-axis is $\log_{10}$ of  the errors.
%The second order methods are in blue, third order in red, fourth order in green, fifth order in magenta,
% sixth order in cyan, seventh order in black, and eighth order in purple.
%The SSPIF-TSRK methods with $s=p$ have a dashed line, those with $s=9$ have a solid line, 
%and those with $(s,p)=(8,7), (11,8)$ have a dotted line. 
The slope of the line is given by $m$ in the legend.
} \label{fig:VDPconv}
\end{figure}  
%In this figure  we show that the $\log_{10}$ of the errors
%in the first component vs. the $\log_{10}$ of the time-step for 
%a selection of eSSPIFRK methods. 
%In this case we used a van der Pol problem that is not highly oscillatory, to  avoid the order reduction that is known to occur with 
%integrating factor methods. This convergence study purposely avoids this issue in order to test the formal 
%convergence of the generated methods.  

% \begin{figure}
%  \begin{minipage}[c]{0.57\textwidth}
% \label{fig:WENOBurgers_adv1}  \includegraphics[scale=.325]{Comparing3rdOrderTVD_ex3b.eps} 
%   \end{minipage}\hfill
%  \begin{minipage}[c]{0.425\textwidth}
%    \caption{ Example 4 with $a=5$.
%On the x-axis is the value of $\lambda = \frac{\dt}{\dx}$,
%on the y axis is $\log_{10}$ of the maximal rise in TV.
%The red solid line is the eSSPIFRK(3,3) method,
%the blue dotted line is the IFRK method based on the eSSPRK(3,3) Shu-Osher method,
%the black dash-dot line is the ETDRK3 method, and the green dashed line is the 
%third order explicit  eSSPRK(3,3)  Shu-Osher method}
%  \end{minipage}
%\end{figure}

\subsection{Example 2:  Sharpness of SSP  time-step for a linear problem}\label{sec:Num_linear}
While the importance of SSP methods is not limited to  cases where we need to preserve  
the TVD property,  TVD studies tend to exhibit the sharpness of the SSP time-step
and so are traditionally used to test SSP time-stepping methods. 
We begin with a case where the TVD property of the spatial discretization can be proven, and 
investigate the TVD behavior of the numerical solutions resulting from 
evolving the semi-discretized problem using our SSP integrating factor two-step Runge--Kutta methods.

We use our SSPIF-TSRK methods to evolve in time the linear advection equation 
on the domain $x\in [0,1]$ with periodic boundary conditions and a step function initial condition
\begin{align}\label{lineartestproblem}
u_t + a u_x + u_x & = 0 \hspace{.75in}
    u(0,x)  =
\begin{cases}
1, & \text{if } \frac{1}{4} \leq x \leq \frac{3}{4} \\
0, & \text{else }.% \nonumber
\end{cases}
\end{align}
Using  a simple  first-order forward difference  to discretize  the spatial derivatives 
with a spatial grid of $1000$ points. 
We split the problem as with $L u \approx - a u_x$ and $ N(u) \approx - u_x$.
The spatial discretization is TVD when coupled with forward Euler in the sense that
\[ \|u^n + \Delta t L u^n \|_{TV} \leq \| u^n \|_{TV}  \; \; \; \mbox{for}. \; \; \;  \tDtFE  = \frac{1}{a} \Dx \]
and 
\[ \|u^n + \Delta t N( u^n)  \|_{TV} \leq \| u^n \|_{TV}  \; \; \; \mbox{for} \; \; \;  \DtFE  =  \Dx. \]

We evolve the numerical solution forward ten time-steps for different values of $\lambda = \frac{\dt}{\Dx}$,
and measure the total variation at each stage. To ensure internal stage monotonicity we 
 compare the total variation at any stage to the  total variation at the previous stage.
We are interested in the  {\em observed TVD time-step} $\dt^{TVD}_{obs}$, which is
the size of  time-step $\dt$ at which the maximal rise  in total variation between any two stages is greater than $10^{-12}$. We compare the observed TVD time-step with the theoretical TVD time-step condition. 
The SSP coefficient that corresponds to   the value of the observed TVD time-step is called the 
 {\em observed SSP coefficient} $\sspcoef_{obs}$, defined by
  \[ \sspcoef_{obs} =  \frac{\dt^{TVD}_{obs}}{\DtFE}  .\]
  Since the forward Euler TVD time-step for this problem is $ \DtFE  =  \Dx$, we have
 \[ \sspcoef_{obs} = \frac{\dt^{TVD}_{obs}}{\Dx} =  \lambda^{TVD}_{obs}.\]

\subsubsection{Example 2a: Comparison of integrating methods for wavespeeds $\boldsymbol{a=5}$ and $\boldsymbol{a=2}$}
In this example we consider a variety of SSPIF-TSRK methods for the time-evolution 
of \eqref{lineartestproblem} with wavespeeds $a=5$ and $a=2$.
 
\begin{figure}[h]
\begin{center}
\includegraphics[scale=.3]{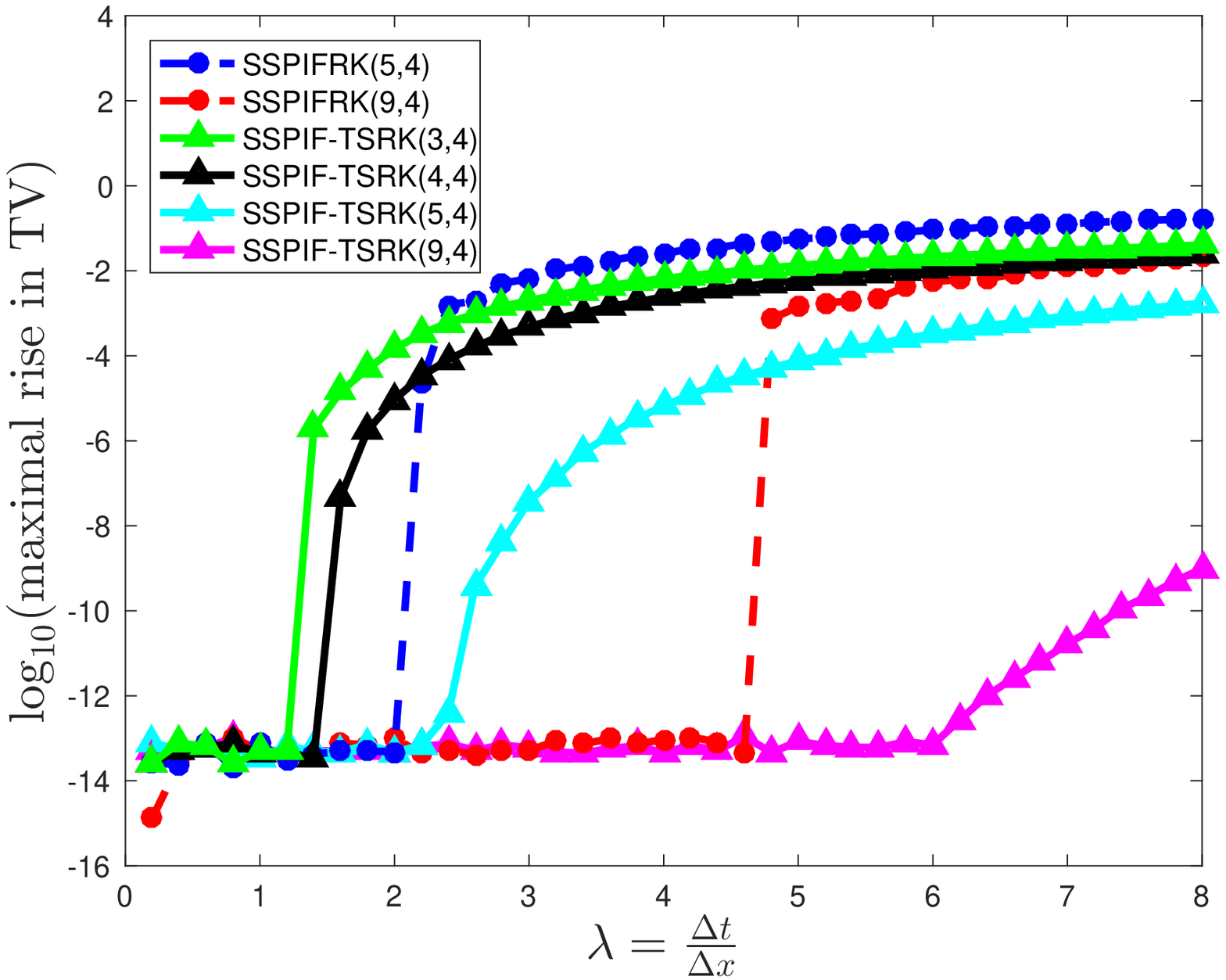}
\includegraphics[scale=.3]{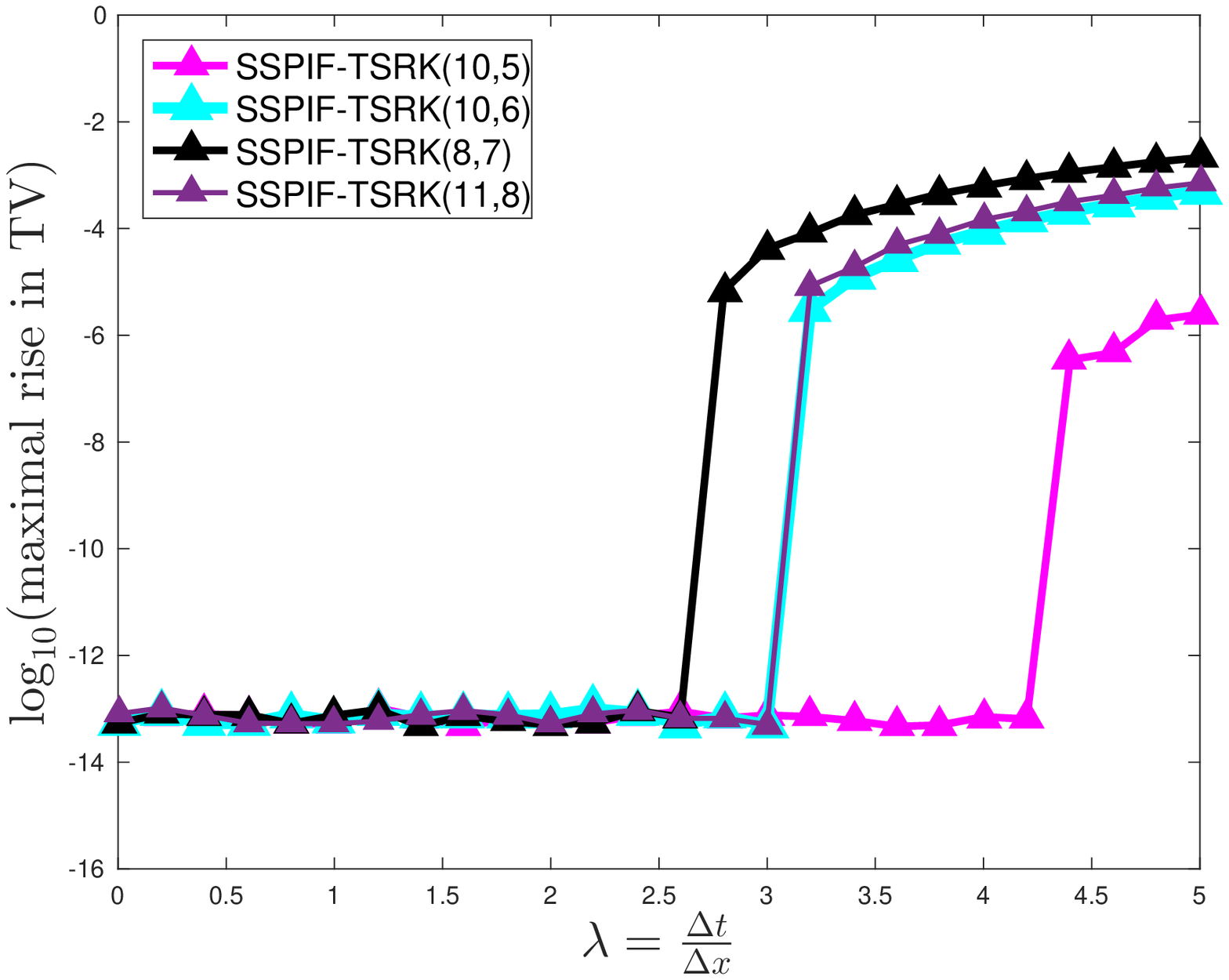} 
\end{center} 
 \caption{Example 2a: Linear advection with a step function initial condition and wavespeeds $a=5$ (left) and $a=2$ (right).
%%    On the x-axis is the value of $\lambda = \frac{\dt}{\dx}$,
%% on the y axis is $\log_{10}$ of the maximal rise in TV.
%The eSSPRK(5,3) method is in blue, the eSSPKG(5,3) in green, the 
% IMEXSSP(5,3,$K=0.1$) method in cyan, and the eSSPIFRK(5,3) in red.
 \label{fig:ex2aTVD} }
\end{figure}
 
In Figure \ref{fig:ex2aTVD} (left) we show the observed maximal rise in total variation when 
this equation is evolved forward by several SSPIF-TSRK methods using different values of the 
CFL number $\lambda = \frac{\dt}{\dx}$. We consider a selection of fourth order methods.
 First, we explore the performance  of the  SSPIF-TSRK$(s,p)$  methods  with $(s,p)=(3,4), (4,4)$
(the corresponding SSPIFRK methods do not exist).
 If we consider the effective observed time-step \[ \sspcoef_{obs}^{eff} = \frac{1}{s} \sspcoef_{obs}\]
(i.e. the observed time-step normalized by the number of stages $s$) we see that the SSPIFRK$(9,4)$ has 
$ \sspcoef_{obs}^{eff} = 0.47$ while the SSPIF-TSRK$(3,4)$ has  a smaller $ \sspcoef_{obs}^{eff} = 0.42$
and SSPIF-TSRK$(4,4)$ has an even smaller $ \sspcoef_{obs}^{eff} = 0.4$. 
However, we observe that the allowable TVD time-step of the SSPIFRK methods with 
$(s,p)=(5,4), (9,4)$ is smaller than that of the  corresponding SSPIF-TSRK methods.
The SSPIF-TSRK methods allow us to obtain higher order for fewer stages than the  SSPIFRK methods, 
and larger observed time-steps for the same $(s,p)$.

Figure \ref{fig:ex2aTVD} (right) shows us the performance of the higher order SSPIF-TSRK methods, 
with orders $p=5,6,7,8$.  In this case we use $a=2$ in \eqref{lineartestproblem}. We selected the best performing methods
in terms of observed SSP time-step, the SSPIF-TSRK$(s,p)$ methods for $(s,p)=(10,5), (10,6), (8,7), (11,8)$.
We note that the $(s,p)=(11,8)$ method has the same SSP time-step performance as the $(s,p)=(10,6)$, showing the
eighth order method to be a viable and efficient method.

%\begin{figure}[h]
%\begin{center}
%\includegraphics[scale=.375]{linearTVD4thOrderV2.eps}
%\end{center} 
% \caption{Example 2a: Linear advection with a step function initial condition and $a=5$.
%%%    On the x-axis is the value of $\lambda = \frac{\dt}{\dx}$,
%%% on the y axis is $\log_{10}$ of the maximal rise in TV.
%%The eSSPRK(5,3) method is in blue, the eSSPKG(5,3) in green, the 
%% IMEXSSP(5,3,$K=0.1$) method in cyan, and the eSSPIFRK(5,3) in red.
% \label{fig:ex2aTVD} }
%\end{figure}

\subsubsection{Example 2b: Considering different wavespeeds}
In \cite{SSPIFRK-SINUM} we showed that unlike fully explicit SSP Runge--Kutta methods or implicit-explicit SSP Runge--Kutta (SSP-IMEX) methods, the SSP integrating factor Runge--Kutta 
schemes have an allowable SSP time-step that is not impacted by the restriction on the time-step resulting from the linear component $L$. This is also true for the SSP integrating factor two-step Runge--Kutta methods. 
In this section, we investigate how different wavespeeds in $L$ impact the allowable SSP time-step of different methods.

 \begin{table}[h] 
 \begin{center}
 \begin{tabular}{|l|c|ccc|} \hline
 Method  &   $\sspcoef $  &    \multicolumn{3}{|c|}{$\lambda^{TVD}_{obs}$} \\ 
 & 			& for $a=0$ &  $a=1.0$ &  $a=5$ \\ \hline
SSPIF-TSRK(3,4) & 0.8588 & 1.0454 & 1.2550 & 1.2621  \\
SSPIF-TSRK(5,4) & {\bf  2.3523} & {\bf 2.3523} & {\bf 2.3523} & 2.4123  \\
SSPIF-TSRK(9,4) & {\bf 5.2120} & {\bf 5.2120} & {\bf 5.2120} & 6.4010 \\ 
SSPIF-TSRK(4,5) & 0.8542 & 1.1852 & 1.3388 & 1.3389 \\
SSPIF-TSRK(6,5) & 2.3093 & 2.3093 & 2.3093 & 2.3094 \\
SSPIF-TSRK(9,5) & {\bf 3.9426} & {\bf 3.9426 } & {\bf 3.9426} & 4.1173  \\
SSPIF-TSRK(6,6) & 0.5958 & 1.7771 & 1.7891 & 1.7893 \\
SSPIF-TSRK(7,6) & 1.2671 & { 2.0239} & { 2.0239} & 2.0261 \\
SSPIF-TSRK(9,6) & 2.4784 & { 2.8038} & { 2.8038}  & 2.8204\\
SSPIF-TSRK(8,7) & 0.5666 & 1.6624 & 2.6737 & 2.7788 \\
SSPIF-TSRK(9,7) & 1.0715 & 2.1626 & 2.4053 & 2.4053 \\ 
SSPIF-TSRK(11,8)& 0.2743 & 2.3871 & 3.1137 & 3.1271 \\ \hline
 \end{tabular}
 \end{center}
 \caption{ \label{tab:IFTSex} The observed SSP coefficient of a variety of SSPIF-TSRK$(s,p)$
 methods
 compared to theirpredicted SSP coefficient, shown for  various wavespeeds $a$.
  The value of $a$ does not negatively impact the observed SSP coefficient. }
 \end{table}

In Table \ref{tab:IFTSex} we show the observed SSP coefficient $\lambda^{TVD}_{obs}$
for values of $a=0,1,5$ in \eqref{lineartestproblem} for a variety of SSPIF-TSRK methods. We notice that
for the methods with $(s,p)=(5,4), (9,4), (6,5), (9,5)$ the observed SSP coefficient for $a=0,1$ was 
sharp, i.e. exactly the SSP coefficient $\sspcoef$ predicted by the theory. As $a$ gets larger ($a=5$) 
the  observed SSP coefficient becomes larger than predicted by the theory, which only provides a
guaranteed lower bound. This occurs because the
damping effect of the exponential increases as $a$ gets larger and so rises in total variation are 
damped out and fall below the threshold. For the other methods, the  observed SSP coefficient for $a=0$ 
is larger than predicted by the theory. For $(s,p)= (7.6), (9,6)$ the observed SSP coefficient  for 
 $a=0,1$ was the same, while for $a=5$ the observed SSP coefficient  is larger.
 For the other methods, we see an increase in the observed SSP coefficient as $a$ increases.
 
 These results are in sharp distinction to explicit SSP Runge--Kutta methods and SSP-IMEX methods.
In Table \ref{tab:ex2b} we compare the observed SSP coefficient of some
 four stage explicit SSP integrating factor methods
 to those of the explicit SSP Runge--Kutta  method eSSPRK(4,3) and the 
 SSP-IMEX(4,3,K) methods. 
 
When the eSSPRK(4,3) method is applied to \eqref{lineartestproblem} with 
 wavespeed $a+1$, where $a=0,1,2,10$,  Table \ref{tab:ex2b} shows that the observed  
 SSP coefficient exactly matches the predicted
 \[ \lambda^{TVD}_{obs} =   \frac{\sspcoef }{a+1}   = \frac{2}{a+1}.\]
 This means that as the wavespeed $a$ increases, the observed SSP coefficient goes down as
 expected.
 
 The SSP-IMEX(4,3,K) we use are from \cite{SSPIMEX}, and 
have SSP explicit and implicit parts that  were optimized for the SSP step size 
for each value of $K=\frac{1}{a}$.  The results presented in Table \ref{tab:ex2b}
show that, as expected from SSP theory, the observed SSP coefficient
decays linearly as the wavespeed $a$ rises. This matches with the 
results in \cite{SSPIMEX} that, in contrast to the fact that IMEX methods may be
A-stable, they cannot be unconditionally SSP past first order. This limitation
highlights the need for SSP integrating factor methods.

Finally, we see that the observed SSP coefficients for the 
explicit SSP integrating factor methods
%SSPIF-TSRK(4,3), SSPIF-TSRK(4,4), and eSSPIFRK(4,3) 
do not decay as the wavespeed $a$ goes up. For the 
 SSPIFRK(4,3) method the observed SSP coefficient is exactly the same for 
$a=0,1,2,10$, while the SSPIF-TSRK(4,3) and SSPIF-TSRK(4,4) methods
have the same observed SSP coefficient for $a=0,1,2$ and a {\em larger} 
observed SSP coefficient for $a=10$.

\begin{table}[h] 
 \begin{center}
 \begin{tabular}{|l|cccc|} \hline
 Method    &  \multicolumn{4}{|c|}{$\lambda^{TVD}_{obs}$} \\ 
			&  $a=0$ &   $a=1.0$ & $a=2$  &$a=10$ \\ \hline
SSPIF-TSRK(4,3) & 2.303 & 2.303 & 2.303 & 2.775  \\
SSPIF-TSRK(4,4) & 1.593 & 1.593 & 1.593 & 1.639    \\ %&a=10 & a=20 & a=100
SSPIFRK(4,3)       & 1.818 & 1.818 & 1.818 & 1.818  \\ %& 1.818 & 1.818 & 4.200
SSP-IMEX(4,3,K)     & 2.000 & 1.476 & 1.192 & 0.310\\ %&0.310 & 0.162 & 0.033
eSSPRK(4,3) 	       & 2.000 & 1.000 & 0.666 & 0.181 \\ \hline %& 0.181 & 0.0952 & 0.019
 \end{tabular}
 \end{center}
 \caption{ \label{tab:ex2b}The observed SSP coefficients for integrating factor methods, an SSP-IMEX method, 
 and an explicit Runge--Kutta method for Example 2b with various wavespeeds $a$. 
 The value of $a$ does not negatively impact the observed SSP coefficient
 for the integrating factor methods, but it does for the IMEX and the explicit Runge--Kutta methods.  }
 \end{table}

\subsection{Example 3: Sharpness of SSP time-step for a nonlinear problem}\label{sec:Num_nonlinear}
Once again we consider the motivating problem \eqref{BurgersAdvection}
in Subsection \ref{motivating}:
  \begin{align}
u_t + a u_x +  \left( \frac{1}{2} u^2 \right)_x & = 0 \hspace{.75in}
    u(0,x)  =
\begin{cases}
1, & \text{if } 0 \leq x \leq 1/2 \\
0, & \text{if } x>1/2 \nonumber
\end{cases}
\end{align}
on the domain $[0,1]$ with periodic boundary conditions.

We used a first-order upwind difference to semi-discretize, with $400$ spatial points, the linear term $Lu \approx - a u_x$, 
and a fifth order WENO finite difference for the  nonlinear terms $N(u) \approx -  \left( \frac{1}{2} u^2 \right)_x$. 
We note that the WENO scheme is not TVD. However, its total variation is generally well-controlled in 
typical simulations.

We evolve the semi-discrerized problem forward 25 time steps
using $\dt = \lambda \dx$, and measure the total variation at each stage. We then  calculate the 
maximal rise in total variation over each stage.  Once again, the quantity 
of interest is the observed TVD time-step,  but in the absence of a forward Euler condition
(as WENO is not TVD) we cannot determine the observed SSP coefficient, only the observed TVD
time-step.

We first show the results for a selection of third order methods on problem \eqref{BurgersAdvection} 
with the value  $a=5$. In Figure \ref{fig:WENOBurgers_adv} (left) we 
plot  $\log_{10}$ of the maximal rise in total variation
versus the ratio $\lambda = \frac{\dt}{\Delta x}$. Our
SSPIF-TSRK(3,3) method out-performs the 
SSPIFRK(3,3) method in \cite{SSPIFRK-SINUM} by more than 10\%. 
In contrast, the fully explicit three stage third order Shu-Osher method applied 
begins to feature a large rise in total variation
for a much smaller value of $\lambda=.15 \approx \frac{1}{1+a}$, as expected.
Evolving the transformed problem with the Shu-Osher method \eqref{SOIF} (denoted 
IF with eSSPRK(3,3)) and with the  three-stage third order Cox and Matthews exponential time differencing
 Runge--Kutta method   (ETDRK3)  \cite{CoxMatthews2002} 
results in a maximal rise in total variation that increases rapidly with $\lambda$.

\begin{figure}[h]
\begin{center}
  \includegraphics[scale=.3]{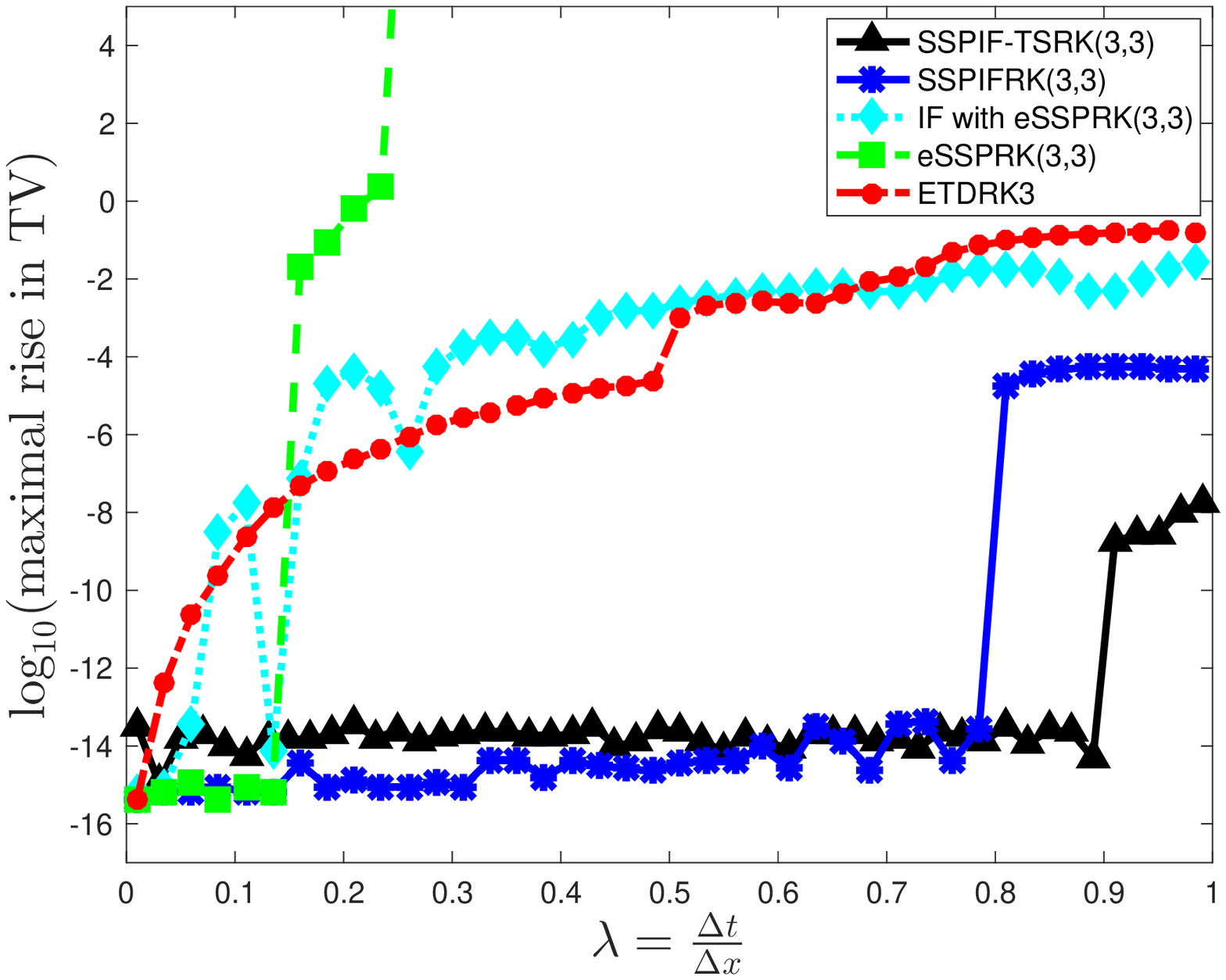} 
  \includegraphics[scale=.3]{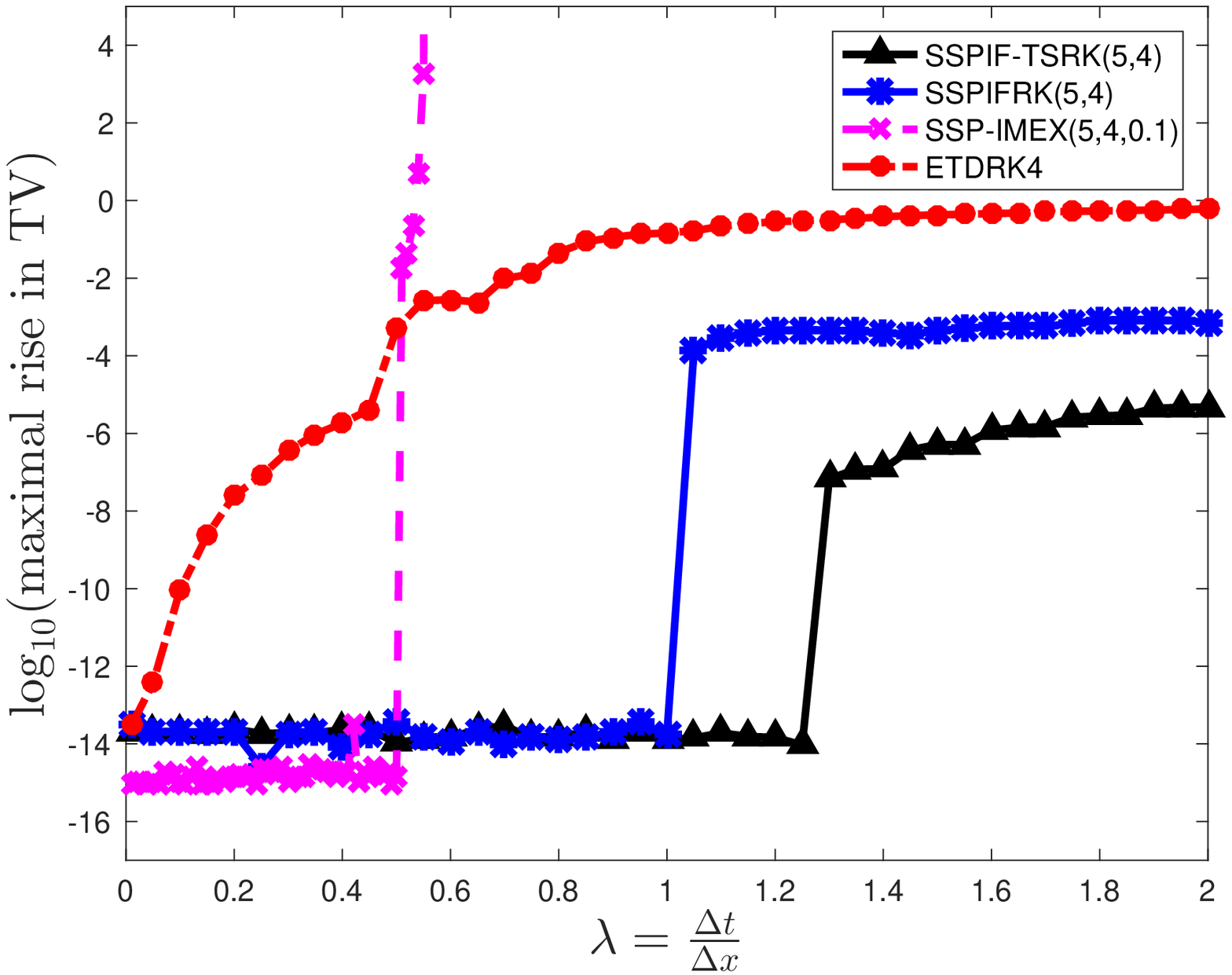} 
    \caption{ 
%     On the x-axis is the value of $\lambda = \frac{\dt}{\dx}$,
%on the y axis is $\log_{10}$ of the maximal rise in TV at each stage,
%for a selection of methods solving  \eqref{nonlineartestproblem} for $a=5$. 
A comparison of the TVD time-step for a variety of integrating factor, exponential time differencing,
explicit and IMEX  SSP Runge--Kutta methods applied to Example 3.
Left:    Third order methods. Right: Fourth order methods.
 \label{fig:WENOBurgers_adv} }
 \end{center}
\end{figure}

In Figure \ref{fig:WENOBurgers_adv}  (right) we show a similar study using 
fourth order methods. 
Our SSPIF-TSRK(5,4) method allows a  TVD time-step of $\lambda_{obs} \approx 1.3$, 
while  the  corresponding integrating factor Runge--Kutta method
SSPIFRK(5,4) method maintains a  small maximal rise in total variation 
until   $\lambda_{obs} \approx 1$. The fourth order ETDRK4 method  of  \cite{CoxMatthews2002}  
does not have good TVD performance:
the maximal rise in total variation when using the 
ETDRK4  rises rapidly from the smallest value of $\lambda$.
The  SSP-IMEX(5,4,0.1) method features an observed  value of $\lambda_{obs} \approx 0.5$, 
pointing once again to the fact that IMEX methods are not as  suitable when the SSP property 
is desired.

Finally, we compare TVD time-step of the SSPIF-TSRK$(s,p)$ methods with $(s,p)=(4,3),(5,4), (6,5), (7,6), (8,7)$.
In Figure \ref{fig:WENOBurgers_advOrders} we show the  $\log_{10}$ of the maximal rise in total variation 
at each stage compared to the  previous stage, plotted against  the ratio $\lambda = \frac{\dt}{\Delta x}$. 
These integrating factor methods all perform well for this problem.
Notably, the seventh order method has the largest observed TVD time-step and the sixth order method 
has the smallest observed TVD time-step for both $a=1$ and $a=5$. The fourth order method
behaves consistently for both $a=1$ and $a=5$. These results are all consistent with those of the 
linear test case.
Once again, we observe that as the exponent increases we generally have  a larger allowable TVD time-step, 
as we noted above.

\begin{figure}[h]
\begin{center}
  \includegraphics[scale=.3]{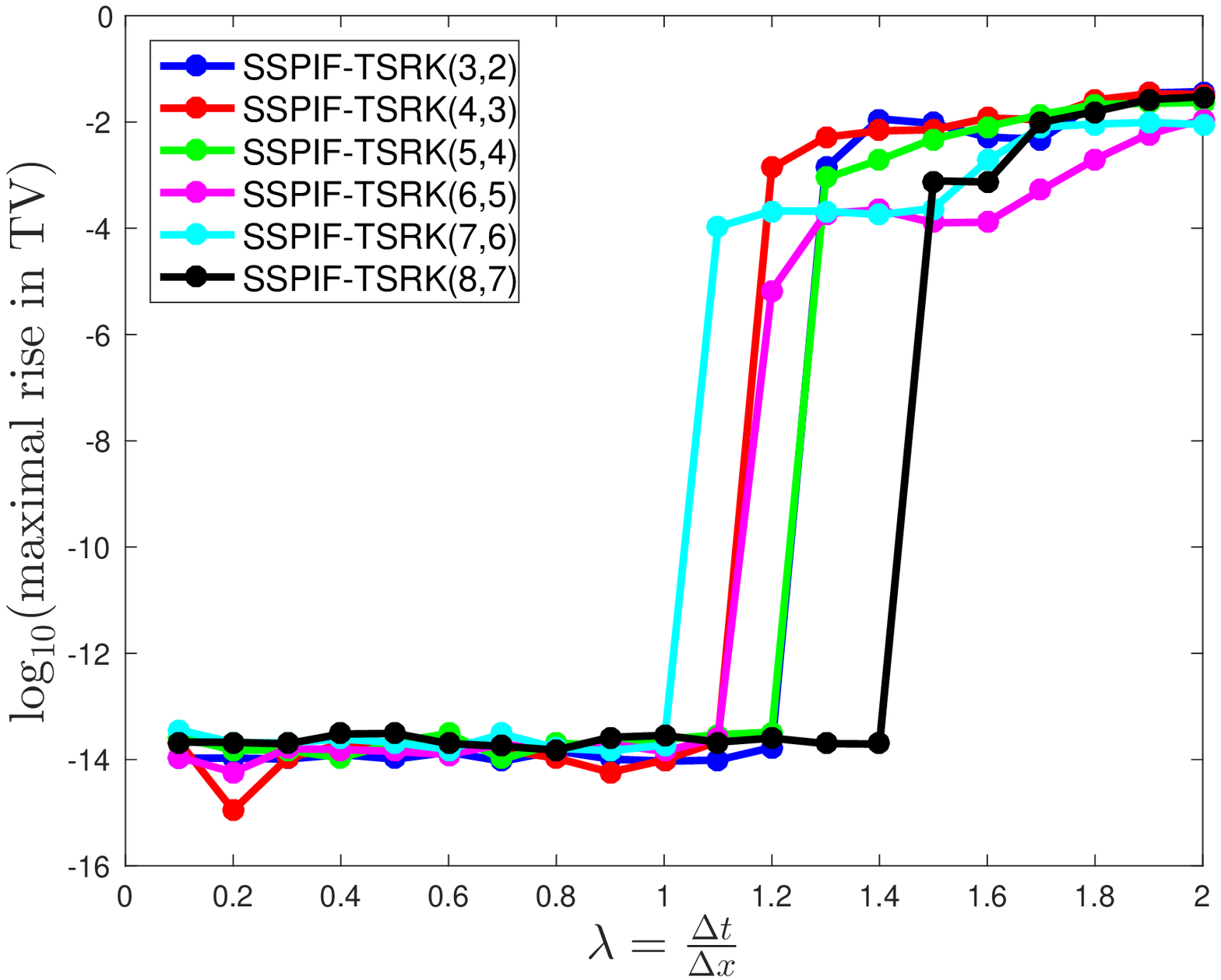} 
  \includegraphics[scale=.3]{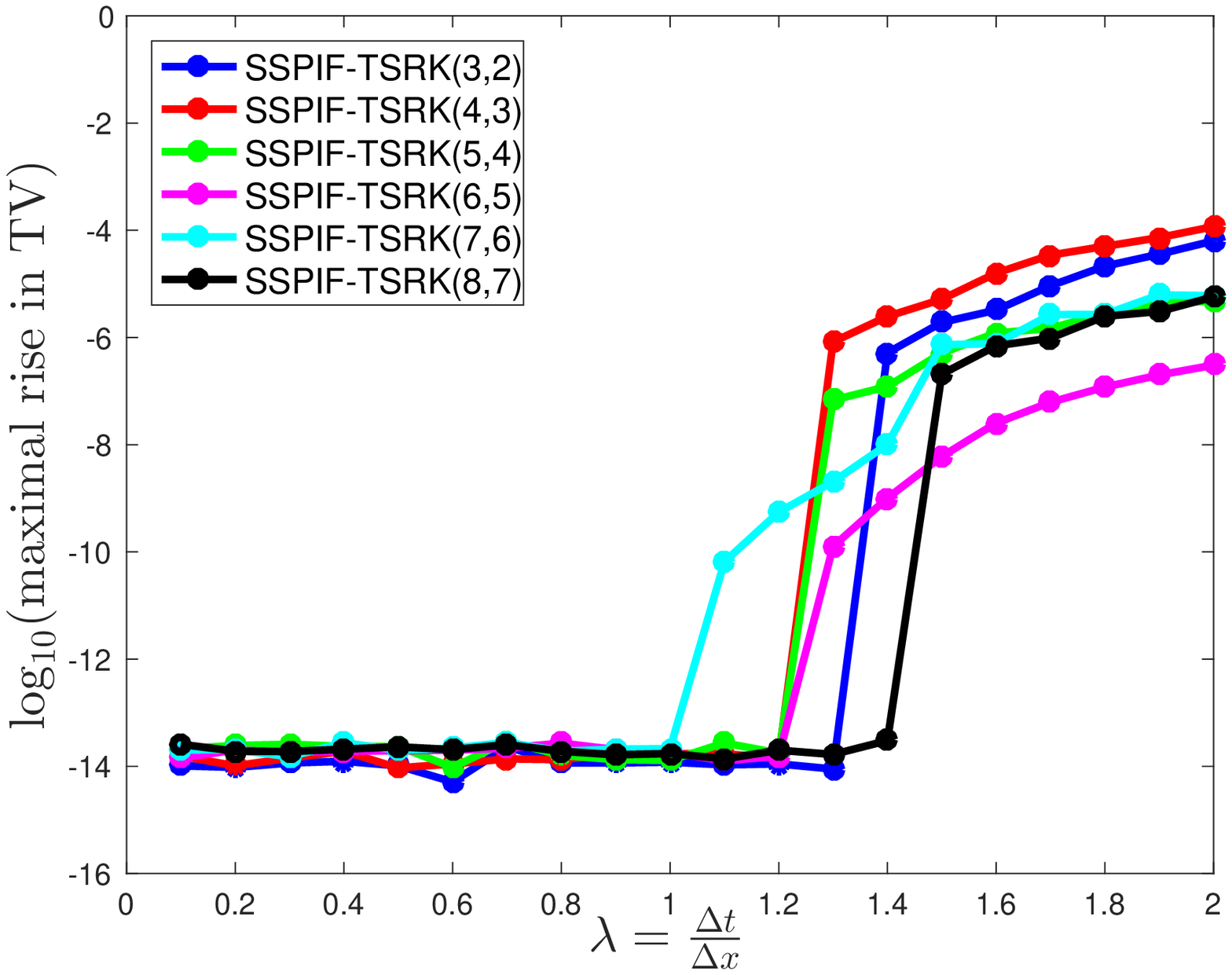} 
    \caption{Performance of integrating factor methods for Example 3 for $a=1$ (left) and $a=5$ (right).
 \label{fig:WENOBurgers_advOrders} }
 \end{center}
\end{figure}

\section{Conclusions}  \label{sec:conclusions}
In this work we extended the SSP theory for integrating factor Runge--Kutta methods  established in \cite{SSPIFRK-SINUM}
to integrating factor two step Runge--Kutta methods. SSP integrating factor methods are needed for problems that 
require a nonlinear non-inner product stability property, and where the linear component is severely restricting the time-step. Unlike the case where linear $L_2$ stability is of interest, implicit or IMEX methods do not alleviate 
this restriction. However,  integrating factor Runge--Kutta methods based on applying a SSP 
Runge--Kutta method with nondecreasing abscissas to the transformed equation completely  alleviate the time-step restriction. SSP Runge--Kutta methods only exist up to fourth order. To extend these results to higher order methods, we considered herein  SSP two-step Runge--Kutta methods applied in conjunction with an integrating factor approach. We showed that, just as for the Runge--Kutta case, if the two-step Runge--Kutta methods are SSP and have non-decreasing abscissas then they can be used to time-step the transformed problem. We enhanced the optimization problem in \cite{tsrk} with the condition \eqref{nnabscissa} and found SSP two-step Runge--Kutta methods with non-decreasing abscissas that have optimized SSP coefficients (presented in Table  \ref{tab:SSPcoefIF}). We tested the resulting SSPIF-TSRK methods on a variety of problems 
and showed that the new high order methods perform as expected in terms of convergence and the strong stability property.

\bigskip

\appendix

\section{Order Conditions}  \label{orderconditions}
%\section{Order conditions}
%{\bf Put in OC}
%{\bf Edit this!}
%\textcolor{red}{The equality constraints \eqref{eq:constraints3} for the optimization problem above come from the order conditions, which were derived in \cite{Butcher}.}he order conditions (22c)
 The order conditions, \eqref{eq:constraints3}  for a $k$-step SSP Runge-Kutta were derived by Ketcheson
and used in \cite{msrk}.
We can write the generalized form for a multi-step Runge--Kutta method as 
\begin{align*} 
y_1^n & =  u^n \\
y_i^n & =   \sum_{l=1}^{k} d_{il} u^{n-k+l} + \Dt \sum_{l=1}^{k-1}\hat{a}_{il} F(u^{n-k+1}) + \Dt\sum_{j=1}^{i-1} a_{ij} F(y_j^n) \; \; \; \;  2 \leq i \leq s \\
u^{n+1} & = \sum_{l=1}^{k} \theta_l u^{n-k+l}  + \Dt\sum_{l=1}^{k-1} \hat{b}_{l} F(u^{n-k+l}) + \Dt\sum_{j=1}^s b_j F(y_j^n).
\end{align*} 
and rewrite it into a matrix form by defining
\begin{align*} 
\tilde{\mD} & = \begin{pmatrix}\mI_{(k-1) \times (k-1)} \ \mzero_{1 \times (k-1)} \\ \mD \\   \end{pmatrix} \ \ \ \
\tilde{\mA}   = \begin{pmatrix}\mzero & \mzero  \\ \mAh & \mA  \\  \end{pmatrix} \ \ \ \
\text{and} & \tilde{\bb} =  \begin{pmatrix} \bbh & \bb  \end{pmatrix}. 
\end{align*}
The method then becomes
%\begin{subequations} %\label{eq:mrk}
\begin{align*} 
\by^n & = \tilde{\mD} \bu^n + \Delta t  \tilde{\mA} \bff^n \\
u^{n+1} & = \btheta\transpose \bu^n + \Delta t \tilde{\bb}\transpose \bff^n
\end{align*}
%\end{subequations}
and by leting $\vl$ be the vector $\vl=\left(k-1, k-2, ..., 1,0 \right)^T$ to compute the abscissas 
$\vc = \tilde{\mA} \ve -\tilde{\mD} \vl$
to get the the following expressions 
\begin{align*} 
\ste_\rho& = \frac{1}{\rho!} \left(\vc^{\rho}-\tilde{\mD}(-\vl)^\rho\right)-\frac{1}{(\rho-1)!}\tilde{\mA}\vc^{\rho-1}\ \ \\ 
\tau_\rho &= \frac{1}{\rho!} \left(1-\btheta(-\vl)^\rho\right)-\frac{1}{(\rho-1)!}\tilde{\bb}^T\vc^{\rho-1} \ \
\end{align*}

Below are the order conditions for methods from first to eighth order. As the order increases the methods must satisfy all previous order conditions as well as the additional ones for each order.

\noindent First:
\begin{align*}
\tilde{\vb^T} \ve & = 1+\btheta^T\vl
\end{align*}
Second:
\begin{align*} %\label{secondorder}
\tilde{\vb^T} \vc & =  \frac{1-\btheta^T\vl^2}{2}.
\end{align*}
Third:
\begin{align*} %\label{thirdorder}
\tilde{\vb}^T \left(\vc^2 \right) & = \frac{1+\btheta^T\vl^3}{3}, &
\tilde{\vb}^T \ste_2 & = 0.
%\vb^T\mA\vc=\frac{1}{6}.
\end{align*}
Fourth:
\begin{align*} % \label{fourthorder}
\tilde{\vb}^T \left(\vc^3 \right)  = \frac{1-\btheta^T\vl^4}{4}, \quad \quad  &
\tilde{\vb}^T \tilde{\mA}\ste_2 = 0, &
\tilde{\vb}^T \mC\ste_2 = 0, \quad \quad &  
\tilde{\vb}^T \ste_3 = 0  
\end{align*}
Fifth:
\begin{align*} % \label{fourthorder}
\tilde{\vb}^T \left(\vc^4\right)  = \frac{1+\btheta^T\vl^5}{5},\quad \quad&
\tilde{\vb}^T \tilde{\mA}\ste_3 = 0, & 
\tilde{\vb}^T \mC\ste_3 = 0, \quad \quad &  
\tilde{\vb}^T \ste_4 = 0, & 
\ste_2 = 0 
\end{align*}
Sixth:
\begin{align*} % \label{fourthorder}
\tilde{\vb}^T \left(\vc^5 \right)  = \frac{1-\btheta^T\vl^6}{6},\quad \quad&
\tilde{\vb}^T \tilde{\mA}\ste_4 = 0, &
\tilde{\vb}^T \mC\ste_4 = 0,\quad \quad &
\tilde{\vb}^T \ste_5 = 0, \\
\tilde{\vb}^T \tilde{\mA^2}\ste_3 = 0,\quad \quad & 
\tilde{\vb}^T \tilde{\mA}\mC\ste_3 = 0, &
\tilde{\vb}^T \mC\tilde{\mA}\ste_3 = 0,\quad \quad &
\tilde{\vb}^T \mC^2\ste_3 = 0
\end{align*}
Seventh:
\begin{align*} % \label{fourthorder}
\tilde{\vb}^T \left(\vc^6 \right)  = \frac{1+\btheta^T\vl^7}{7}, \quad \quad \quad \quad & 
\tilde{\vb}^T \tilde{\mA}\ste_5 = 0, & 
\tilde{\vb}^T \mC\ste_5 = 0,  \\
\tilde{\vb}^T \ste_6 = 0, \quad \quad \quad \quad& 
\tilde{\vb}^T \tilde{\mA^2}\ste_4 = 0, & 
\tilde{\vb}^T \tilde{\mA}\mC\ste_4 = 0,  \\
\tilde{\vb}^T \mC\tilde{\mA}\ste_4 = 0, \quad \quad \quad \quad &  
\tilde{\vb}^T \mC^2\ste_4 = 0, &
\ste_3 = 0
\end{align*}
Eighth:
\begin{align*} % \label{fourthorder}
\tilde{\vb}^T \left(\vc^7 \right)  = \frac{1-\btheta^T\vl^8}{8}, \quad \quad& 
\tilde{\vb}^T \tilde{\mA}\ste_6 = 0, & 
\tilde{\vb}^T \mC\ste_6 = 0, \quad \quad&
\tilde{\vb}^T \ste_7 = 0, \\
\tilde{\vb}^T \tilde{\mA^3}\ste_4 = 0, \quad \quad & 
\tilde{\vb}^T \tilde{\mA^2}\ste_5 = 0, & 
\tilde{\vb}^T \tilde{\mA^2}\mC\ste_4 = 0,\quad \quad & 
\tilde{\vb}^T \tilde{\mA}\mC\tilde{\mA}\ste_4 = 0, \\
\tilde{\vb}^T \tilde{\mA}\mC\ste_5 = 0, \quad \quad&  
\tilde{\vb}^T \tilde{\mA}\mC^2\ste_4 = 0, &
\tilde{\vb}^T \mC\tilde{\mA^2}\ste_4 = 0, \quad \quad &
\tilde{\vb}^T \mC\tilde{\mA}\ste_5 = 0, \\
\tilde{\vb}^T \mC\tilde{\mA}\mC\ste_4 = 0, \quad \quad &  
\tilde{\vb}^T \mC^2\tilde{\mA}\ste_4 = 0, &
\tilde{\vb}^T \mC^2\ste_5 = 0, \quad \quad&
\tilde{\vb}^T \mC^3\ste_4 = 0\\
\end{align*}
Where $\mC$ is a diagonal matrix of the abscissas $\mC=diag(\vc)$ and the exponentiation is considered to be element wise i.e. $\vc^3=\left(\vc \cdot \vc \cdot \vc \right)$.

\bigskip

{\bf Acknowledgment.} 
This publication is based on work supported by  AFOSR grant FA9550-18-1-0383. 
A part of this research is sponsored by the Office of Advanced Scientific Computing Research; US Department of Energy, and was performed at the Oak Ridge National Laboratory, which is managed by UT-Battelle, LLC under Contract no. De-AC05-00OR22725. This manuscript has been authored by UT-Battelle, LLC, under contract DE-AC05-00OR22725 with the US Department of Energy. The United States Government retains and the publisher, by accepting the article for publication, acknowledges that the United States Government retains a non-exclusive, paid-up, irrevocable, world-wide license to publish or reproduce the published form of this manuscript, or allow others to do so, for United States Government purposes.


\newpage

\begin{thebibliography}{10}

\bibitem{AlMohyHigham}
{\sc A.H. Al Mohy and  N.J. Higham},
{\em Computing the action of the matrix exponential of a matrix,
with an application to exponential integrators},
SIAM Journal on Scientific Computing 33(2) (2011), pp.~488--511.

\bibitem{Sandu} %paper on MSRK
{\sc E. M. Constantinescu and A. Sandu}, 
{\em Optimal strong-stability-preserving general linear methods},
 SIAM Journal of Scientific Computing 32(5) (2010), pp.~3130-3150.



\bibitem{Grant1}
{\sc A.J. Christlieb, S. Gottlieb, Z. Grant, and D. C. Seal},
{\em Explicit Strong Stability Preserving Multistage Two-Derivative Time-Stepping Schemes},
Journal of Scientific Computing 68(3) (2016),  pp.~914-942.



\bibitem{SSPIMEX} 
{\sc S. Conde, S. Gottlieb, Z. Grant, J.N. Shadid},
{\em Implicit and Implicit-Explicit Strong Stability Preserving RungeÐKutta Methods with High Linear Order},
 Journal of Scientific Computing 73(2)  (2017), pp.~667--690.


\bibitem{CoxMatthews2002}
{\sc S.~Cox and P.~Matthews}, {\em Exponential time differencing for stiff
  systems}, Journal of Computational Physics, 176 (2002), pp.~430--455.

\bibitem{ferracina2008}
{\sc  L. Ferracina and M.N. Spijker},
{\em  Strong stability of singly-diagonally-implicit {R}unge-{K}utta methods},
Applied Numerical Mathematics 58 (2008), pp.~1675-1686.


\bibitem{GaudreaultRainwaterTokman}
{\sc   S. Gaudreault, G. Rainwater, M. Tokman},
{\em KIOPS: A fast adaptive Krylov subspace solver for exponential integrators},
Journal of Computational Physics 372  (2018), pp.~236-255. 


\bibitem{SSPIF-TSRKgithub}
{\sc S.~Gottlieb, Z.~Grant, and L.~Isherwood}, {\em Optimized strong stability
  preserving integrating factor two-step {R}unge--{K}utta methods}.
\newblock \url{https://github.com/SSPmethods/SSPIF-TSRK-methods}.

\bibitem{SSPbook2011}
{\sc S.~Gottlieb, D.~I. Ketcheson, and C.-W. Shu}, {\em Strong Stability
  Preserving Runge--Kutta and Multistep Time Discretizations}, World Scientific
  Press, 2011.

\bibitem{gottliebshu1998}
{\sc S.~Gottlieb and C.-W. Shu}, {\em Total variation diminishing Runge--Kutta
  methods}, Mathematics of Computation, 67 (1998), pp.~73--85.

\bibitem{gottlieb2001}
{\sc S.~Gottlieb, C.-W. Shu, and E.~Tadmor}, 
{\em Strong Stability Preserving  High-Order Time Discretization Methods}, 
  SIAM Review, 43 (2001), pp.~89--112.

\bibitem{Grant2}
{\sc Z. Grant, S. Gottlieb, D.C. Seal},
{\em A Strong Stability Preserving Analysis for Explicit Multistage Two-Derivative 
Time-Stepping Schemes Based on Taylor Series Conditions},  
To appear in Communication on Applied Mathematics and Computation 1(1) (2019), pp.~21-59. %https://arxiv.org/abs/1804.10526


\bibitem{HesthavenCLbook}
{\sc J.S. Hesthaven},
{\em Numerical methods for conservation laws: From analysis to algorithms},
 SIAM Publishing, Philadelphia (2017).

\bibitem{hundsdorfer2003}
{\sc W. Hundsdorfer, S.J. Ruuth and R.J. Spiteri},
{\em Monotonicity-preserving linear multistep methods},
SIAM Journal on Numerical Analysis 41 (2003), pp.~605--623.


\bibitem{SSPIFRK-SINUM}
{\sc L. Isherwood, S. Gottlieb, Z. Grant},  
{\em Strong Stability Preserving Integrating Factor Runge--Kutta Methods.}
SIAM Journal on Numerical Analysis 56(6) (2018), pp.~3276--3307.

\bibitem{SSPIFRK-downwind}
{\sc L. Isherwood, S. Gottlieb, Z. Grant},
{\em Downwinding for Preserving Strong Stability in Explicit Integrating Factor Runge--Kutta Methods},
 Pure and Applied Mathematics Quarterly 14(1) (2019), pp.~3-25.  %https://arxiv.org/abs/1810.04800

\bibitem{WENO}
{\sc G.-S. Jiang and C.-W. Shu},
{\em Efficient Implementation of Weighted ENO Schemes},
Journal of Computational Physics 126(1) (1996), pp.~202-228.



\bibitem{ketcheson2008}
{\sc D.~I. Ketcheson}, {\em Highly efficient strong stability preserving
  {R}unge--{K}utta methods with low-storage implementations}, SIAM Journal on
  Scientific Computing, 30 (2008), pp.~2113--2136.

\bibitem{ketcheson2009}
{\sc D.I. Ketcheson, C.B. Macdonald and S. Gottlieb},
{\em Optimal implicit strong stability preserving {R}unge-{K}utta methods},
Applied Numerical Mathematics 52 (2009), pp.~373--392.


\bibitem{ketcheson2009a}
 {\sc D. I. Ketcheson},
 {\em Computation of optimal monotonicity preserving general linear methods},
 Mathematics of Computation 78 (2009), pp.~1497--1513.

\bibitem{ketcheson2011}
 {\sc D. I. Ketcheson},
{\em Step sizes for strong stability preservation with downwind-biased operators},
SIAM Journal on Numerical Analysis 49 (4) (2011), pp.~1649--1660.

\bibitem{tsrk}
{\sc D.I. Ketcheson, S. Gottlieb, and C. B. Macdonald},
{\em Strong stability preserving two- step Runge-Kutta methods},
 SIAM Journal on Numerical Analysis 49 (2012), pp.~2618-2639.
 
 \bibitem{msrk}
{\sc  C. Bresten, S. Gottlieb, Z. Grant, D. Higgs, D.I. Ketcheson, and A. Nemeth},  
{\em Explicit strong stability preserving multistep Runge-Kutta methods},
  Mathematics of Computation 86 (2017), pp.~747-769.
 

\bibitem{kraaijevanger1991}
{\sc J.~F. B.~M. Kraaijevanger}, 
{\em Contractivity of {R}unge--{K}utta methods}, 
BIT, 31 (1991), pp.~482--528.


\bibitem{lenferink1991}
{\sc H.W.J. Lenferink},
{\em Contractivity-preserving implicit linear multistep methods},
Mathematics of Computation 56 (1991), pp.~177--199.


\bibitem{LeVequeBook}
{\sc R. J. LeVeque},
{\em Numerical Methods for Conservation Laws},
ETH Lectures in Mathematics Series, Birkhauser-Verlag, (1990).

\bibitem{NiesenWright}
{\sc J. Niesen and W.M. Wright},
{\em Algorithm 919: A Krylov subspace algorithm for evaluating the $\phi$-functions appearing
in exponential integrators},
ACM Transactions on Mathematical Software 38(3) (2012), pp.~22. 





\bibitem{ruuth2001}
{\sc S.~J. Ruuth and R.~J. Spiteri}, {\em Two barriers on
  strong-stability-preserving time discretization methods}, Journal of
  Scientific Computation, 17 (2002), pp.~211--220.


\bibitem{shu1988b}
{\sc C.-W. Shu}, {\em Total-variation diminishing time discretizations}, SIAM
  Journal on Scientific Statistical Computing 9 (1988), pp.~1073--1084.

\bibitem{shu1988}
{\sc C.-W. Shu and S.~Osher}, {\em Efficient implementation of essentially
  non-oscillatory shock-capturing schemes}, Journal of Computational Physics
  77 (1988), pp.~439--471.


\bibitem{Sidje}
{\sc R.B. Sidje},
{\em EXPOKIT: A software package for computing matrix exponentials},
ACM Transactions on Mathematical Software 24(1) (1998), pp.~130--156.


\bibitem{spijker1983}
{\sc M.N. Spijker},
{\em Contractivity in the numerical solution of initial value problems},
Numerische Mathematik 42 (1983), pp.~271--290.

\bibitem{spijker2007}
{\sc M.~Spijker}, {\em Stepsize conditions for general monotonicity in
  numerical initial value problems}, SIAM Journal on Numerical Analysis, 
  45 (2008), pp.~1226--1245.

\bibitem{SpiteriRuuth2002}
{\sc R.~J. Spiteri and S.~J. Ruuth}, {\em A new class of optimal high-order
  strong-stability-preserving time discretization methods}, SIAM Journal on
  Numerical Analysis, 40 (2002), pp.~469--491.





\end{thebibliography}
\end{document}